%% file: prob_bounds_siopt.tex
\documentclass[final,onefignum,onetabnum]{siamart190516}

\input{ex_shared}

\ifpdf
\hypersetup{
  pdftitle={Extremal Probability Bounds},
  pdfauthor={D. Padmanabhan et. al.}
}
\fi

\usepackage{bbm}
\usepackage{bm}
\usepackage{subcaption}
\usepackage{tikz}

\newcommand{\argmin}{\operatornamewithlimits{argmin}}
\def\texitem#1{\par\smallskip\noindent\hangindent 25pt
               \hbox to 25pt {\hss #1 ~}\ignorespaces}
\newcommand{\mb}[1]{\mbox{\boldmath $#1$}}
\usepackage{tabu}
\usepackage{cleveref}

\usepackage{tikz}
\usepackage{etoolbox}
\usetikzlibrary{hobby,backgrounds,shapes,positioning,calc}
\definecolor{darkred}{RGB}{139,0,0}
\definecolor{darkgreen}{RGB}{0,139,0}

\preto\equation{\par\nobreak\small\noindent}
\preto\align{\par\nobreak\small\noindent}
\expandafter\preto\csname align*\endcsname{\par\nobreak\small\noindent}
\begin{document}

\maketitle

\begin{abstract}
 In this paper, we compute the tightest possible bounds on the probability that the optimal value of a combinatorial optimization problem in maximization form with a random objective exceeds a given number, assuming only knowledge of the marginal distributions of the objective coefficient vector. The bounds are ``extremal'' since they are valid across all joint distributions with the given marginals. We analyze the complexity of computing the bounds assuming discrete marginals and identify instances when the bounds are computable in polynomial time. For compact 0/1 V-polytopes, we show that the tightest upper bound is weakly NP-hard to compute by providing a pseudopolynomial time algorithm. On the other hand, the tightest lower bound is shown to be strongly NP-hard to compute for compact 0/1 V-polytopes by restricting attention to Bernoulli random variables. For compact 0/1 H-polytopes, for the special case of PERT networks arising in project management, we show that the tightest upper bound is weakly NP-hard to compute by providing a pseudopolynomial time algorithm. The results in the paper complement existing results in the literature for computing the probability with independent random variables.
\end{abstract}
\vspace{-0.25cm}
\vspace{-0.25cm}
\vspace{-0.25cm}
\section{Introduction}
In this paper, we are interested in the random combinatorial optimization problem of the form:
\begin{equation}
\label{eqn:Z_c}
\begin{array}{rllll}
\displaystyle Z(\tilde{\bold{c}})  =  {\max} & \displaystyle\tilde{\bold{c}}'\bold{x} \\
 \mbox{s.t.} & \displaystyle \bold{x} \in \mathcal{X} \subseteq \{0, 1\}^n,
\end{array}
\end{equation}
where $\tilde{\bold{c}} = (\tilde{c}_1, \ldots, \tilde{c}_n)$ is an $n$-dimensional random vector and $\mathcal{X}$ is a subset of the set $\{0,1\}^n $. Our main goal is to compute bounds on the probability that the random optimal value $Z(\tilde{\bold{c}})$ is greater than or equal to a fixed number $r$ when the marginal distributions of the random variables $\tilde{c}_i$ for $i \in [n]$ are specified. Throughout the paper, given a nonnegative integer $n$, we let $[n]$ denote the set $\{1, \ldots, n\}$ and given integers $n_1 \leq n_2$, we let $[n_1,n_2]$ denote the set $\{n_1,\ldots, n_2 \}$. We assume each random variable $\tilde{c}_i$ is discrete with marginal probabilities specified as $\mathbb{P}(\tilde{c}_i = c_{ik}) = p_{ik}$ for $k \in [0,K]$ and support given by ${\cal C}_i = \{c_{i0},\ldots,c_{ik}\}$ where the values are ordered as $c_{i0} < \ldots < c_{iK}$. The marginal probabilities satisfy $\sum_{k}p_{ik} = 1$ for all $i \in [n]$ and $p_{ik} \geq 0$ for all $i \in [n]$ and $k \in [0,K]$. Let $\Theta$ denote the set of all joint distributions on $\tilde{\bold{c}}$ consistent with the marginal distributions:
 \begin{align*}
\displaystyle \Theta  = \left\{ \theta \in \mathbb{P} (\prod_{i=1}^{n}{\cal C}_i): \mathbb{P}_{\theta}(\tilde{c}_{i} = c_{ik}) = p_{ik}, \text{ for }i \in [n], k \in [0,K]\right\},
\end{align*}
where $\mathbb{P} (\prod_{i=1}^{n}{\cal C}_i)$ is the set of all joint distributions supported on the set ${\cal C}_1 \times \ldots \times {\cal C}_n$.
Given a fixed value $r $, we are interested in computing the following extremal probability bounds:
\begin{align*}
\mbox{(Upper bound) } U(r) =  \max_{\theta \in \Theta} \mathbb{P}_{\theta}(Z(\tilde{\bold{c}}) \geq r), \\
\mbox{(Lower bound) } L(r) =  \min_{\theta \in \Theta} \mathbb{P}_{\theta}(Z(\tilde{\bold{c}}) \geq r).
\end{align*}
A related probability of interest to compute is when the random combinatorial optimization problem has mutually independent random variables in the objective coefficient vector. Specifically, let $\theta_{ind}$ be the joint distribution:
 \begin{align*}
\displaystyle \mathbb{P}_{\theta_{ind}}(\tilde{c}_{1} = c_{1k_1},\ldots,\tilde{c}_{n} = c_{nk_n}) = p_{1k_1}\times \ldots \times p_{nk_n}, \text{ for }k_1 \in [0,K],\ldots,k_n  \in [0,K],
\end{align*}
where $\theta_{ind} \in \Theta$. The probability for the independent distribution is given as:
\begin{align*}
\mbox{(Independence) } I(r) =  \mathbb{P}_{\theta_{ind}}(Z(\tilde{\bold{c}}) \geq r),
\end{align*}
where $U(r) \geq I(r) \geq L(r)$. We discuss the complexity of computing $U(r), L(r)$ and $I(r)$ in this paper.



\subsection{Applications}
 Our interest in studying these probability bounds are motivated from the applications discussed next.
 
(a) In simple settings, the extremal probability bounds discussed in this paper reduce to well known probability bounds. For example, consider computing an upper bound on the probability of occurrence of at least one of the $n$ events $E_1, \ldots, E_n$. If only the probabilities of occurrence of each individual event is known, Boole's union bound given by $\min(\sum_{i} \mathbb{P}(E_i),1)$ is tight. This bound arises as a special case of the framework above, by defining the Bernoulli random variables as $\tilde{c}_i = 1$ if $E_i$ occurs and $\tilde{c}_i = 0 $ otherwise, setting $Z(\tilde{\bold{c}}) = \sum_{i} \tilde{c}_i$ and $r = 1$. Bounds on the sum of random variables when only the marginal distributions are given has been extensively studied in the risk, insurance and finance settings; see Chapter 4 in \cite{ruschbook3}.
%

(b) In the context of  Program Evaluation and Review Technique (PERT) networks, the distribution of the completion time of a project needs to be estimated where the project is composed of several activities with random activity times \cite{elmaghraby1977activity}. Planning decisions are made taking into account the distribution of the project completion time. In this setting, $Z(\tilde{\bold{c}})$ is the optimal value of a longest path problem on a directed acyclic graph where the arc length vector $\tilde{\bold{c}}$ denotes the random activity duration vector. The probability of the completion time exceeding a deadline $r$ is a relevant measure of the performance of the project (higher the probability, worse the performance). Much of the literature has looked at computing this probability under the assumption of independence or limited dependence among the activity durations  \cite{Fulkerson1962,Dodin1985,Hagstrom1988,Ball1995,Mohring2001}.
However in PERT networks, there is evidence of significant dependence occurring among the activity durations when the resources are shared across activities or when adverse events affect all activities \cite{Ringer1971}.
This motivates the interest in the computation of extremal probability bounds.

(c) In the context of reliability, the probability of a system being functional is characterized in terms of the probabilities of the subcomponents being operational. Extremal probability bounds then provide an estimate of the robustness of the system to dependence among the subcomponents; see the book of \cite{waykuo}. For example, the $s$-$t$ reliability measure (probability that there exists at least one operational path from node $s$ to node $t$ in a graph) is computed by assuming each edge $(i,j)$ on the graph is associated with a Bernoulli random variable $\tilde{c}_{ij}$ where $\tilde{c}_{ij} = 1$ if the arc is operational and $0$ if it fails and formulating $Z(\tilde{\bold{c}})$ as a minimum $s$-$t$ cut problem with $r=1$.


\subsection{Existing Results and Contributions of This Paper}
Evaluating $Z(\bold{c})$ is already NP-hard for the class of deterministic combinatorial optimization problems.
In this paper we focus on combinatorial optimization problems where the convex hull of the feasible region has a compact representation and $Z(\bold{c})$ is computable in polynomial time. Two representations we consider are described next:

(a) V-polytope: The convex hull of the set $\mathcal{X} \subseteq \{0,1\}^n$, denoted by $\mbox{conv}(\mathcal{X})$, is given by a convex combination of a set of $P$ points:
    \begin{equation}
\label{eqn:V}
\begin{array}{rllll}
\displaystyle \mbox{conv}(\mathcal{X}) & = & \displaystyle \mbox{conv}\{\bold{x}^1, \ldots, \bold{x}^P \},\\
& = & \displaystyle \left\{\sum_{j=1}^{P}\lambda_j\bold{x}^j : \sum_{j=1}^{P}\lambda_j = 1, \lambda_j \geq 0, \text{ for }j \in [P]\right\},
\end{array}
\end{equation}
where $\bold{x}^1, \ldots, \bold{x}^P \in \{0,1\}^n$. In this representation, $P$ is typically exponential in $n$ and so (\ref{eqn:V}) is only useful when $P$ is allowed to be part of the input size specification. The size of the input instance for computing $U(r)$ or $L(r)$ in this case is given by: $$\displaystyle \mbox{Size of input } = O(\max(K,P)n\max(\log_2 U_1,\log_2 U_2)),$$
where $K+1$ is an upper bound on the size of any marginal support, $n$ is the number of random variables, $P$ is the number of extreme points in the V-polytope, $U_1$ and $U_2$ are the maximum numerical values among the integers in the ratio representation of the rational numbers $p_{ik}$ and $c_{ik}$ across all $i$ and $k$. The logarithmic dependence of the input size on the magnitude of the input probabilities and the support points arises since $O(\log_2 U)$ binary digits are needed to represent a positive integer $U$.\\
 For $P = 1$ and $\bold{x} = \bold{1}_n$ (the vector of all ones), we get $Z(\tilde{\bold{c}}) = \sum_{i} \tilde{c}_i$.
Even for the sum of random variables, computing $U(r)$ and $L(r)$ have been shown to be NP-hard for two point marginal distributions \cite{Kreinovich2006} using a reduction from the partition problem. Computing $I(r)$ with two point marginal distributions has also shown to be \#P-hard \cite{kleinberg} using a reduction from the problem of counting the number of feasible solutions to a 0-1 knapsack problem. In special cases, the bounds are efficiently computable. These include the sum of $n = 2$ random variables \cite{makarov,ruschendorff1982} where simple formulas exist for arbitrary distributions and for the sum of $n$ random variables with $K = 1$ (Bernoulli random variables) \cite{ruger1978}. Many other bounds, not necessarily tight have also been proposed in the literature (see Chapter 4 in \cite{ruschbook3} for several such bounds). \\
We add to this stream of results by showing that for compact 0/1 V-polytopes, the upper bound $U(r)$ is in fact weakly NP-hard to compute by providing a pseudopolynomial time algorithm. Specifically, we show that when the random variables take values $c_{ik} = k$ for $k \in [0,K]$, it is possible to compute $U(r)$ by solving a linear program that is of polynomial size in $K$, $n$, $P$ and $\log_2(U_1)$. The key aspect of this result is that dependence on the parameter $U_2$ is overcome. Furthermore for Bernoulli random variables, we provide further reduction in the polynomial size of the linear program for computing $U(r)$. On the other hand, we show the lower bound $L(r)$ is strongly NP-hard to compute. Specifically, we show that it is not possible to compute $L(r)$ in polynomial time in the input size even when the random variables are Bernoulli, unless P = NP. We also provide a \#P-hardness  result for independent Bernoulli random variables in this representation.


(b) H-polytope: The convex hull of the set $\mathcal{X} \subseteq \{0,1\}^n$ is given by: 
    \begin{equation}
\label{eqn:H}
\begin{array}{rllll}
\displaystyle \mbox{conv}(\mathcal{X}) = \{\bold{x} : \bold{A}\bold{x} \leq \bold{b}\},
\end{array}
\end{equation}
where the matrix $\bold{A}$ is of size $m \times n$ and $\bold{b}$ is a vector of length $m$. In this representation, the size of the input instance for computing $U(r)$ or $L(r)$ is given by:
\[ O(\max(K,m)n\max(\log_2 U_1,\log_2 U_2,\log_2 U_3)),\]
where in addition to the other parameters, $U_3$ is the maximum numerical value among the integers in the ratio representation of the rational numbers in the matrix $\bold{A}$ and vector $\bold{b}$.
An example of a combinatorial optimization problem with a compact 0/1 H-polytope representation is a PERT network where computing $Z(\bold{c})$ is possible in polynomial time. In PERT networks, the extreme points are characterized by the $s$-$t$ paths in the network which can be exponentially large. The V-polytope representation is not useful in this setting. However $Z(\bold{c})$ can be computed efficiently using a linear program which grows polynomially in the size of the network characterized by the number of nodes and edges in the graph, rather than the number of paths in the graph. Computing $I(r)$ is however known to be NP-hard for PERT networks even when the activity durations are Bernoulli random variables \cite{Hagstrom1988}. For certain classes of reliability problems, polynomial time computable bounds $U(r)$ and $L(r)$ have been proposed in the literature \cite{Zemel1982,Weiss1986}. However these formulations make use of the equivalence of separation and optimization \cite{grotschel2012geometric} to prove polynomial time complexity bounds without providing compact formulations that are easy to implement in practice.\\
We add to the stream of results in H-polytopes by showing that that for PERT networks a polynomial sized linear program can be used to compute the tightest upper bound $U(r)$ when the activity durations are restricted to take values in $[0,K]$. In turn, this shows that for PERT networks, the upper bound $U(r)$ is weakly NP-hard. This provides the maximum (worst case) probability of the random project completion time exceeding a given deadline. 

A related area of research is distributionally robust chance constraints \cite{xie2019optimized,Hanasusanto2017} wherein the constraints of an optimization problem are required to be satisfied with high probability. The difference of this line of research from our work is that we instead focus on computing the tail probabilities of  the objective value of an uncertain optimization problem. 

The structure of the paper is as follows. In \Cref{sec:v_polytope} and  \Cref{sec:h_polytope} respectively, we provide results for the V-polytope and the H-polytope. Numerical results provided in \Cref{sec:numericals} compare various probability bounds in random walks and PERT networks. We also show applications in models exhibiting limited dependence.

\section{Bounds for the V-Polytope}
\label{sec:v_polytope}
\subsection{Upper Bound}
We begin by developing a pseudopolynomial time algorithm for computing $U(r)$ for 0/1 V-polytopes. The bound is computed using a linear program. For the analysis, we assume that the support of each random variable $\tilde{c}_i$ is contained in ${\cal C}_i = [0,K]$. Under this restriction on support, we are looking for algorithms with running time polynomial in $K$, $n$, $P$ and $\log_2(U_1)$ thereby dropping the explicit dependence on the size of the input required to represent the marginal support values $c_{ik}$.
The support of the random vector is contained in $[0,K]^n$ which is of size $O(K^n)$. Let us first write an exponential sized LP to compute $U(r)$ (see \cite{hailperin1965}):
\begin{align*}
\begin{array}{rllll}
U(r) = \max  & \displaystyle \sum_{\bold{c} \in  [0,K]^n} \theta(\bold{c})\mathbbm{1}_{\{Z(\bold{c}) \geq r\}} \\
\mbox{s.t.}&\displaystyle  \sum_{\bold{c} \in  [0,K]^n} \theta(\bold{c})= 1,\\
&\displaystyle  \sum_{\bold{c} \in  [0,K]^n: c_i= k} \theta(\bold{c})= p_{ik}, \text{ for } i \in [n], k \in [0,K],\\
&\displaystyle  \theta(\bold{c})\geq 0, \text{ for } \bold{c} \in  [0,K]^n,
\end{array}
\end{align*}
where $\mathbbm{1}_{\{Z(\bold{c}) \geq r\}} = 1$ if $Z(\bold{c}) \geq r$ and $0$ otherwise and the decision variables are the joint probabilities $\theta(\bold{c}) = \mathbb{P}(\tilde{\bold{c}} = \bold{c})$ for $\bold{c} \in [0,K]^n$. The primal linear program has a polynomial number of constraints but an exponential number of variables. From strong duality,  $U(r)$ is the optimal value of the corresponding dual linear program,
\begin{align}
U(r) = \min & \; \; \displaystyle\lambda + \sum_{i=1}^n \sum_{k =0}^K \alpha_{ik} p_{ik} \nonumber\\
\mbox{s.t.}&\; \;  \displaystyle\lambda + \sum_{i=1}^n \sum_{k =0}^K \alpha_{ik} \mathbbm{1}_{\{c_i = k\}} \geq 1, \text{ for } Z(\bold{c}) \geq r, \bold{c} \in  [0,K]^n, \label{constr:1_general}\\
& \; \; \displaystyle \lambda + \sum_{i=1}^n \sum_{k=0}^K \alpha_{ik} \mathbbm{1}_{\{c_i = k\}}  \geq 0, \text{ for } \bold{c} \in  [0,K]^n,\label{constr:0_general}
\end{align}
where the decision variables are $\lambda$ and $\alpha_{ik}$ for $i \in [n]$ and $k \in [0,K]$. The dual linear program has a polynomial number of variables but an exponential number of constraints. By the equivalence of separation and optimization \cite{grotschel2012geometric}, a polynomial time algorithm to solve the underlying separation problem for the dual linear program implies the existence of a polynomial time algorithm to compute $U(r)$. We now show that the separation problems corresponding to the constraints \eqref{constr:1_general} and \eqref{constr:0_general} can be solved efficiently and develop a compact linear program to compute $U(r)$.
\begin{theorem}
\label{thm:probability_bound_general_discrete}
Let ${\cal X} = \{\bold{x}^1, \ldots, \bold{x}^P \} \subseteq \{0,1\}^n$. Given the marginal distributions of the random vector $\tilde{\bold{c}}$ as $\mathbb{P}(\tilde{c}_i = k) = p_{ik}$ for $k \in [0,K]$ and $i \in [n]$, the tightest upper bound is computable by solving the linear program:
\begin{align*}
\begin{array}{rllll}
U(r)= \max  & \displaystyle \sum_{\bold{x} \in \mathcal{X}} a_{\bold{x}} \\
\mbox{s.t.}& \displaystyle \sum_{\bold{x} \in \mathcal{X}} a_{\bold{x}} + w = 1, \\
 & \displaystyle h_{ik} + \sum_{\bold{x} \in \mathcal{X}} g_{ik\bold{x}}( 1- x_i) + \sum_{l = k}^{nK} \delta_{ikl\bold{x}} x_i = p_{ik}, \text{ for } i \in [1,n], k \in [0,K],\\
& \displaystyle \sum_{k \in [0, K]} h_{ik} = w , \text{ for } i \in [1, n],\\
& \displaystyle \sum_{k \in [0, K]} g_{ik\bold{x}} = a_{\bold{x}} , \text{ for } i \in [1, n], \bold{x} \in \mathcal{X},\\
&\displaystyle \sum_{k \in [0, K]} \tau_{l\bold{x}} =a_{\bold{x}} , \text{ for } l \in [r, nK], \bold{x} \in \mathcal{X},\\
& \displaystyle \sum_{k=0}^{\min(K,l)} \delta_{ikl\bold{x}} = \sum_{k=0}^{M} \delta_{i+1,k, l+k, \bold{x}} x_{i+1} + \sum_{k=0}^{\min(K,l)} \delta_{i+1,k,l,\bold{x}} (1-x_{i+1})\\
 &\displaystyle \;\;\;\;\;\; \;\;\;\;\;\; \;\;\;\;\;\; \;\;\;\;\;\;  \text{ for }i \in [b_{\bold{x}}+1,n], \text{ for } l \in [0, nK] , \bold{x} \in \mathcal{X}, \\
 &\displaystyle \delta_{b_\bold{x}, l, l, \bold{x}} = \sum_{k=0}^{\min(K,l)} \delta_{b_\bold{x} + 1, k, l+k, \bold{x}} \bold{x}_{b_\bold{x} + 1} + \sum_{k=0}^{\min(K,l)}\delta_{b_\bold{x} + 1, k, l, \bold{x}} (1-x_{b_x + 1}),  \\
 & \displaystyle \;\;\;\;\;\; \;\;\;\;\;\; \;\;\;\;\;\; \;\;\;\;\;\;  \text{ for } l \in [0, K] ,\bold{x} \in \mathcal{X}, \\
 &\displaystyle \boldsymbol{\delta},\bold{g}, \bold{h}, \bold{a}, w, \boldsymbol{\tau} \geq 0,
 \end{array}
\end{align*}
where for every $\bold{x} \in {\cal X}$, $b_{\bold{x}}$ denotes the smallest value of $i \in [n]$ for which $x_i = 1$ and $M=  \min(nK-l, k)$. Specifically the linear program is solvable in time polynomial in $K$, $n$, $P$ and $\log_2(U_1)$.
\end{theorem}

\begin{proof}
We derive the LP by reformulating constraints \eqref{constr:1_general} and \eqref{constr:0_general}. \\
\textbf{Step (1):} Reformulating constraints \eqref{constr:1_general}:\\
We can rewrite constraint \eqref{constr:1_general} as:
$\lambda + W(\boldsymbol{\alpha}) \geq 1,$
where $W(\boldsymbol{\alpha})$ is the optimal value of the following 0-1 integer program:
\begin{flalign*}
\begin{array}{rllll}
W(\boldsymbol{\alpha}) = \min  \;& \displaystyle \sum_{i=1}^n \sum_{k =0}^K \alpha_{ik} y_{ik} \\
\mbox{s.t.}&\displaystyle  \max_{\bold{x} \in \mathcal{X}} \sum_{i=1}^n \left(\sum_{k=0}^K ky_{ik}\right)x_i \geq r, \\
& \displaystyle \sum_{k = 0}^K y_{ik} = 1, \text{ for } i \in [n],\\
& \displaystyle y_{ik} \in \{0,1\}, \text{ for } i \in [n], k \in  [0,K].
\end{array}
\end{flalign*}
This is obtained by defining the binary variable $y_{ik}$ as $\mathbbm{1}_{\{c_i = k\}}$.
Towards further simplification,
for any $\bold{x} \in \mathcal{X}$, $\boldsymbol{\alpha} \in \mathbb{R}^{n \times (K+1)}$ and $ r \in [0,nK]$, define $G(\boldsymbol{\alpha}, \bold{x})$ as the optimal value of the following 0-1 integer program:
\begin{align*}
\begin{array}{rllll}
G(\boldsymbol{\alpha}, \bold{x}) = \min  \;&\displaystyle \sum_{i=1}^n \sum_{k=0}^K \alpha_{ik} y_{ik} \\
\mbox{s.t.}& \displaystyle \sum_{i=1}^n \left(\sum_{k=0}^K ky_{ik}\right)x_i \geq r, \\
& \displaystyle \sum_{k = 0}^K y_{ik} = 1, \text{ for } i \in [n], \\
& \displaystyle y_{ik} \in \{0,1\}, \text{ for } i \in [n], k \in [0,K].\\
\end{array}
\end{align*}
Then we have:
\begin{align*}
\lambda + W(\boldsymbol{\alpha}) \geq 1 \iff \lambda + G(\boldsymbol{\alpha}, \bold{x}) \geq 1 \text{ for } \bold{x} \in \mathcal{X}.
\end{align*}
The value $G(\boldsymbol{\alpha}, \bold{x})$ can be rewritten as $\sum_{i=1}^n \sum_{k=0}^K \alpha_{ik} y_{ik}x_i + \sum_{i=1}^n \sum_{k=0}^K\alpha_{ik} y_{ik}(1-x_i) $. For computing $G(\boldsymbol{\alpha}, \bold{x})$, we need to find an optimal assignment from $ [0,K]$ for each $i$ (through the binary variable $\bold{y}$). Let us first focus on the terms in the objective involving indices $i$ where $x_i = 0$. Observe that if $x_i = 0$ for some $i$, we set $y_{ik^*} = 1$ for $ k^* \in \argmin_{k \in  [0,K]} \alpha_{ik}$ at optimality (with ties broken arbitrarily) and $y_{ik} =0$ for all values of $k \neq k^*$. This is clearly optimal since the first constraint $\sum_{i} (\sum_{k} ky_{ik})x_i \geq r$ is unaffected. The contribution made by this assignment to the overall objective $G(\boldsymbol{\alpha}, \bold{x})$ is captured using the following linear program:
\begin{align}
\label{lp:part_2_I}
\begin{array}{rllll}
\displaystyle \max\limits_{\bold{q}} \;&\displaystyle  \left\{\sum_{i=1}^n q_{i,\bold{x}}:  q_{i,\bold{x}}\leq \alpha_{ik}(1-x_i), \text{ for } k \in [0, K], i \in [n]\right\}.
\end{array}
\end{align}
Now let us look at the remaining part of the objective involving indices $i$ where $x_i = 1$. We are in particular interested in solving the integer program:
\begin{flalign}
\label{ip:multiple_choice_knapsack}
\begin{array}{rllll}
 \min  \;&\displaystyle \sum_{i=1}^n \sum_{k=0}^K \alpha_{ik} y_{ik}x_i \\
\mbox{s.t.}&\displaystyle  \sum_{i=1}^n \left(\sum_{k=0}^K ky_{ik}\right)x_i \geq r, \\
&\displaystyle  \sum_{k = 0}^K y_{ik} = 1, \text{ for } i \in [n], \\
&\displaystyle  y_{ik} \in \{0,1\}, \text{ for } i \in [n], k \in [0,K],
\end{array}
\end{flalign}
which is an instance of a multiple choice knapsack problem \cite{kellerer2004multiple}. We next use a dynamic programming reformulation of this problem to develop the linear program.
Let $f_{i,l, \bold{x}}$ denote the optimal value of the subproblem, which only makes optimal assignment for the variables $y_{jk}$ for all $j \in [i]$:
\begin{align*}
\begin{array}{rllll}
f_{i,l, \bold{x}} = \min  \;&\displaystyle \sum_{j=1}^i \sum_{k=0}^K \alpha_{jk} y_{jk}x_j \\
\mbox{s.t.}& \displaystyle \sum_{j=1}^i \left(\sum_{k=0}^K ky_{jk}\right)x_j=l,\\
& \displaystyle \sum_{k = 0}^K y_{jk} = 1, \text{ for } j \in [i], \\
& \displaystyle y_{jk} \in \{0,1\}, \text{ for } j \in [i], k \in [0,K].\\
\end{array}
\end{align*}
We set for each $\bold{x}$, $b_{\bold{x}}$ as the smallest value of the index $i \in [n]$ such that $x_i = 1$.  We must have $f_{b_{\bold{x}},k,\bold{x}} = \alpha_{b_{\bold{x}},k}$ for $k \in [0,K]$.
For $i > b_{\bold{x}}$, if $x_i = 0$, then $f_{i-1,l,\bold{x}}$ gets passed on to $f_{i,l,\bold{x}}$. However if $x_i = 1$, then $f_{i,l,\bold{x}}$ will take the smallest possible value of $f_{i-1,l-k, \bold{x}} + \alpha_{ik}$ out of all possible values of $k \in [0,K]$, and $y_{ik}=1$ for the corresponding $k$. So we have:
\begin{align*}
f_{i,l, \bold{x}} = f_{i-1,l,\bold{x}}(1-x_i)  +\min_{k \in  [0,K]} \left(f_{i-1,l-k, \bold{x}} + \alpha_{ik}\right)x_i .
\end{align*}
Finally, the optimal objective of \eqref{ip:multiple_choice_knapsack} is $\min_{r \leq l \leq nK}f_{n,l,x}$. Putting together the dynamic programming recursion gives us the following linear program:
\begin{align}
\label{lp:multiple_choice_knapsack}
\begin{array}{rllll}
\max \;& t_{\bold{x}} \\
\mbox{s.t.}& f_{n,l,x} - t_{\bold{x}} \geq 0 \text{ for } l \in [r, nK],\\
&( f_{i-1, l-k,\bold{x}} + \alpha_{ik})x_i +  f_{i-1, l,\bold{x}} (1-x_i ) -f_{i,l,\bold{x}}  \geq 0, \text{ for } i \in [2,n] ,\\
 & \;\;\;\;\;\; \;\;\;\;\;\; \;\;\;\;\;\; \;\;\;\;\;\;  k \in [0,K],  l \in [k, nK] ,\\
 & f_{b_{\bold{x}},k,\bold{x}} = \alpha_{b_{\bold{x}},k}\text{ for } k \in [0,K]. \\
\end{array}
\end{align}
Further putting together  \eqref{lp:part_2_I} and \eqref{lp:multiple_choice_knapsack} we reformulate $G(\boldsymbol{\alpha}, \bold{x})$ as:
\begin{align*}
\begin{array}{rllll}
G(\boldsymbol{\alpha}, \bold{x})  = \max &\displaystyle t_{\bold{x}}  + \sum_{i=1}^n q_{i,\bold{x}} \\
\mbox{s.t.}& f_{n,l,x} - t_{\bold{x}} \geq 0 \text{ for } l \in [r, nK], \\
&( f_{i-1, l-k,\bold{x}} + \alpha_{ik})x_i +  f_{i-1, l,\bold{x}} (1-x_i ) -f_{i,l,\bold{x}}  \geq 0, \text{ for } i \in [2,n] ,\\
 & \;\;\;\;\;\; \;\;\;\;\;\; \;\;\;\;\;\; \;\;\;\;\;\;  k \in [0,K],  l \in [k, nK] ,\\
 & f_{b_{\bold{x}},k,\bold{x}} = \alpha_{b_{\bold{x}},k}, \text{ for } k \in [0,K], \\
& (1-x_i)\alpha_{ik} -q_{i,\bold{x}} \geq 0, \text{ for } i \in [n], k \in [0,K].
\end{array}
\end{align*}
Forcing $\lambda + G(\boldsymbol{\alpha}, \bold{x})$ to be greater than $1$, provides the following equivalent reformulation of the constraint \eqref{constr:1_general}:
\begin{align*}
\begin{array}{rllll}
&\displaystyle \lambda + t_{\bold{x}} + \sum_{i=1}^n q_{i,\bold{x}} \geq 1, \text{ for } \bold{x} \in \mathcal{X}, \\
 &\displaystyle  (1-x_i)\alpha_{ik} -q_{i,\bold{x}} \geq 0, \text{ for } \bold{x} \in \mathcal{X}, i \in [1,n], k \in [0,K],\\
& \displaystyle f_{n,l,x} - t_{\bold{x}} \geq 0, \text{ for } l \in [r, nK], \text{ for } \bold{x} \in \mathcal{X} ,\\
&\displaystyle  ( f_{i-1, l-k,\bold{x}} + \alpha_{ik})x_i +  f_{i-1, l,\bold{x}} (1-x_i ) -f_{i,l,\bold{x}}  \geq 0, \text{ for } i \in [2,n] ,\\
 & \;\;\;\;\;\; \;\;\;\;\;\; \;\;\;\;\;\; \;\;\;\;\;\;  k \in [0,K], l \in [k, nK] , \bold{x} \in \mathcal{X}, \\
 & f_{b_{\bold{x}},k,\bold{x}} = \alpha_{b_{\bold{x}},k},\text{ for } k \in [0,K],\bold{x} \in \mathcal{X}.
\end{array}
\end{align*}
\textbf{Step (2):} Reformulating constraints \eqref{constr:0_general}\\
Note that enforcing  \eqref{constr:0_general}  boils down to ensuring:
\small
\begin{align*}
\begin{array}{rlll}
\displaystyle \lambda +  \min \left\{\sum_{i=1}^n \sum_{k=0}^K \alpha_{ik} y_{ik}: \sum_{k=0}^K y_{ik} = 1, \text{ for } i \in [n], y_{ik} \in \{0,1\}, \text{ for } i \in [n], k\in [0,K]\right\} \geq 0.
\end{array}
\end{align*}
\normalsize
It is easy to see that the optimal value of the optimization problem is attained by $y_{ik} = 1$ for $k = \argmin_{k \in [0,K]} \alpha_{ik}$ for all $i \in [n]$. Thus the constraint can be reformulated as:
\begin{align*}
\begin{array}{rlll}
\displaystyle \lambda + \max \left\{\sum_{i=1}^n v_i: \alpha_{ik} - v_i \geq 0, \text{ for } i \in [1,n], k \in [0,K] \right\} \geq 0.
\end{array}
\end{align*}
Then integrating all the constraints together gives us the following linear program:
\begin{align*}
\begin{array}{rllll}
\min \;\; & \displaystyle \lambda + \sum_{i=1}^n \sum_{k=0}^{K} \alpha_{ik} p_{ik} \\
\mbox{s.t.}&\displaystyle \lambda + t_{\bold{x}} + \sum_{i=1}^n q_{i,\bold{x}} \geq 1, \text{ for } \bold{x} \in \mathcal{X}, \\
 &\displaystyle  (1-x_i)\alpha_{ik} -q_{i,\bold{x}} \geq 0, \text{ for } \bold{x} \in \mathcal{X}, i \in [1,n], k \in [0,K],\\
&\displaystyle  f_{n,l,x} - t_{\bold{x}} \geq 0, \text{ for } l \in [r, nK], \text{ for } \bold{x} \in \mathcal{X} ,\\
&\displaystyle ( f_{i-1, l-k,\bold{x}} + \alpha_{ik})x_i +  f_{i-1, l,\bold{x}} (1-x_i ) -f_{i,l,\bold{x}}  \geq 0, \text{ for } i \in [2,n] ,\\
 &\displaystyle  \;\;\;\;\;\; \;\;\;\;\;\; \;\;\;\;\;\; \;\;\;\;\;\;  \text{ for }k \in [0,K], l \in [k, nK] ,\bold{x} \in \mathcal{X}, \\
 &\displaystyle  f_{b_{\bold{x}},k,\bold{x}} = \alpha_{b_{\bold{x}},k},\text{ for } k \in [0,K], \bold{x} \in \mathcal{X}, \\
 &\displaystyle  \lambda + \sum_{i=1}^n v_i \geq 0,\\
 &\displaystyle  \alpha_{ik} - v_i \geq 0, \text{ for } i \in [1,n], k \in [0,K] .
 \end{array}
\end{align*}
Taking the dual of this linear program gives us the tight reformulation in the theorem.
\end{proof}
The linear program has a total $O(n^2K^2P)$ variables and $O(n^2 K^2 P)$ constraints. When $P$ is polynomial in $n$, this is a polynomial sized linear program in comparison to the original primal linear program which has $O(K^n)$ variables.
We now consider an application of this bound to the sum of random variables.
\subsection{Application to Sum of Random Variables}\label{subsec:discretesums}
The computation of probability bounds for the sum of dependent random variables has received much attention in the literature. In particular, there have been many upper and lower bounds developed with general marginal distributions (discrete or continuous) in the works of \cite{ruschendorff1982,embrechts2006bounds,Puccetti2012,wang2013bounds,wang_2014,blanchet2021convolution} and the references therein. These bounds are typically generated by choosing appropriate dual feasible solutions and are guaranteed to be tight in special cases \cite{ruschbook3}. Given the hardness results for computing these bounds, it is of interest to find instances where the tight bounds are computable in polynomial time.

We now discuss the application of Theorem \ref{thm:probability_bound_general_discrete} to computing bounds for sums of dependent random variables with discrete marginal distributions.
Let $S(r,K)$ denote the following probability bound:
\small
\begin{align*}
S(r,K) = \max\left\{\mathbb{P}_{\theta}\left(\sum_{i=1}^n \tilde{c}_i \geq r\right):\mathbb{P}_{\theta}(\tilde{c}_i = k) = p_{ik}, \text{ for } k \in [K],i \in [n], 
\theta \in \mathbb{P} ([0,K]^n)\right\}.
\end{align*}
\normalsize
For the case of Bernoulli random variables with $K = 1$ where $p_{i0} = 1-p_i$ and $p_{i1} = p_i$, the tightest upper bound for $r = 1$ is given by Boole's union bound:
$$\displaystyle S(1,1) = \min\left(\sum_{i=1}^{n}p_i,1\right).$$
For more general values of $r \in [n]$, the tightest upper bound for the sum of dependent Bernoulli random variables was computed in closed form by \cite{ruger1978}:
\begin{align}
\label{eqn:univar-bound}
\displaystyle S(r,1) = \min \left( \left(\min_{t \in [0,r-1]} \sum_{i=1}^{n-t} \frac{p_{(i)}}{r-t} \right), 1 \right),
\end{align}
where the marginal probabilities $p_1,p_2,\ldots,p_n$ are ordered as $p_{(1)} \leq p_{(2)} \leq \ldots \leq p_{(n)}$. For the sum of discrete random variables with support in $[0,K]$, directly applying Theorem \ref{thm:probability_bound_general_discrete} brings us to the following corollary which shows that the tightest bound is computable in polynomial time. This adds to the stream of literature on identifying instances where the tightest upper bound is computable in polynomial time.

\begin{corollary}
\label{cor:sums_discrete}
Given the marginal distributions of the random vector $\tilde{\bold{c}}$ as $\mathbb{P}(\tilde{c}_i = k) = p_{ik}$ for $k \in [0,K]$ and $i \in [n]$, the tightest upper bound on the sum exceeding a value $r$ is computable by solving the linear program:
\begin{align*}
\begin{array}{rlll}
\displaystyle S(r,K) = \max \;& \displaystyle a \\
\mbox{s.t.}& \displaystyle a + b = 1, \\
&\displaystyle h_{ik} + \sum_{l=k}^{nK} \delta_{i,k,l} = p_{ik}, \text{ for } i \in [n], k \in [0, K],\\
&\displaystyle  \sum_{k=0}^K h_{ik} = b, \text{ for } i \in [n], \\
&\displaystyle  a = \sum_{l=r}^{nK} \tau_{l},\\
& \displaystyle \tau_l = \sum_{k=0}^{\min(K,l)} \delta_{n,k,l}, \text{ for } l \in [r, nK],\\
& \displaystyle \delta_{n,k,l} = 0, \text{ for } l \in [0, r-1], k \in [0, \min(K,l)] ,\\
&\displaystyle \delta_{1,k,k} = \sum_{k'=0}^{\min(K,nK-k)} \delta_{2,k',k'+k}, \text{ for } k \in [0,K], \\
&\displaystyle  \delta_{2,k,k'+k} = 0, \text{ for } k \in [0,K], k' \in [K+1, nK - k], \\
&\displaystyle  \sum_{k=0}^{\min(K,l)}\delta_{i,k,l}= \sum_{k' =0}^{\min(K, nK-l)} \delta_{i+1, k', l+k'}, \\
&\;\;\;\;\;\;\;\;\;\;\;\text{ for } i \in [2,n-1], l \in [0, nK],\\
&\displaystyle  a,b, \bold{h}, \boldsymbol{\delta}, \boldsymbol{\tau} \geq 0.
\end{array}
\end{align*}
\end{corollary}
Next we describe the construction of the extremal distribution using the optimal solution of the linear program in Corollary \ref{cor:sums_discrete}. Given an optimal solution of the linear program denoted by $a^*,b^*, \bold{h}^*, \boldsymbol{\delta}^*, \boldsymbol{\tau}^*$, an extremal distribution is constructed using the following mixture distribution:
\begin{enumerate}
\item Generate a Bernoulli random variable $\tilde{z}$ with probability $a^*$.
\item If $\tilde{z} = 1$,
\begin{enumerate}
\item Generate $\tilde{c}_1 = k$  with probability $\delta_{1,k,k}/a$.
\item For each $i$ in $[2,n]$, generate $\tilde{c}_i$ as follows:
$$\displaystyle \mathbb{P}\left(\tilde{c}_i = k \ \Big{|} \ \sum_{j=1}^i \tilde{c}_i = l\right) =  \frac{\displaystyle \delta_{i,k,l+k}}{\displaystyle \sum_{k' \in  [0,K]}\delta_{i-1,k',l}}, \text{ for } l \in [0, iK].$$
\end{enumerate}
\item If $\tilde{z} = 0$, generate $\tilde{c}_i = k$ with probability $h_{ik}/b$ independently across all $i \in [n]$.
\end{enumerate}
It is straightforward to check that $\theta^{*}$ is the extremal distribution where the optimal decision variables can be interpreted as: $a^* = \mathbb{P}_{\theta^*}(\sum_{i} \tilde{c}_i \geq r)$, $b^* = \mathbb{P}_{\theta^*}(\sum_{i} \tilde{c}_i < r)$.
Additionally, $h_{ik}^* =  \mathbb{P}_{\theta^*}(\tilde{c}_i = k, \sum_{j=1}^n \tilde{c}_j < r)$, $\tau_l^* = \mathbb{P}_{\theta^*}(\sum_{i=1}^n \tilde{c}_i \geq r, \sum_{i=1}^n \tilde{c}_i = l)$ and $\delta_{i,k,l} = \mathbb{P}_{\theta^*}(\tilde{c}_i = k, \sum_{j=1}^i \tilde{c}_i = l, \sum_{l=1}^n \tilde{c}_l = n)$.

\subsubsection{Reduced Formulations for Bernoulli Random Variables}
In the scenario where the random variables take support in $\{0,1\}$, we show that the size of the linear program in Theorem \ref{thm:probability_bound_general_discrete} can be reduced by employing an alternative approach to tackle the separation problem:
\begin{align}
\label{constr:dual_separation}
\min \left\lbrace \sum_{i=1}^n \alpha_i c_i: Z(\bold{c}) \geq r, \bold{c} \in \{0,1\}^n  \right \rbrace.
\end{align}

\begin{theorem}
\label{thm:general_binary_support}
Let ${\cal X} = \{\bold{x}^1, \ldots, \bold{x}^P \} \subseteq \{0,1\}^n$. Given the marginal distributions of the Bernoulli random vector $\tilde{\bold{c}}$ as $\mathbb{P}(\tilde{c}_i = 1) = 1-\mathbb{P}(\tilde{c}_i = 0) = p_{i}$ for $i \in [n]$, the tightest upper bound is computable by solving the linear program:
\begin{align*}
\begin{array}{rllll}
U(r)= \max \;\; & \displaystyle \sum_{\bold{x} \in \mathcal{X}} a_{\bold{x}}\\
\mbox{s.t} \;\;&\displaystyle  \sum_{\bold{x} \in \mathcal{X}} a_{\bold{x}} + b = 1, \\
&\displaystyle  h_i + \sum_{\bold{x} \in \mathcal{X}} g_{i, \bold{x}} = p_i, \text{ for } i \in [n],\\
&\displaystyle  h_i \leq b, \text{ for } i \in [n],\\
&\displaystyle   g_{i, \bold{x}} \leq  a_{\bold{x}}, \text{ for } i \in [n], \bold{x} \in \mathcal{X},\\
&\displaystyle  r a_{\bold{x}} - \sum_{i: x_i=1}  g_{i, \bold{x}} \leq 0, \text{ for } \bold{x} \in \mathcal{X},\\
&\displaystyle  \bold{a} \geq 0, b \geq 0, \bold{g} \geq 0, \bold{h} \geq 0.
\end{array}
\end{align*}
\end{theorem}
\begin{proof}
Constraint \eqref{constr:0_general} in the exponential sized dual linear program for Bernoulli random variables can be rewritten as follows:
\begin{align*}
\begin{array}{rllll}
&\displaystyle \lambda +\min \left\lbrace \sum_{i=1}^n \alpha_i c_i : c_i \in \{0,1\}^n \right\rbrace \geq 0\\
\iff&\displaystyle   \lambda +\min \left\{ \sum_{i=1}^n \alpha_i c_i : 0 \leq c_i \leq 1,\text{ for } i \in [n] \right\} \geq 0 \\
\iff & \displaystyle  \lambda +\max \left\{  \sum_{i=1}^n -\eta_i : \alpha_i + \eta_i \geq 0,\text{ for } i \in [n], \boldsymbol{\eta} \geq 0\right\} \geq 0 \\ 
 \iff & \displaystyle  \lambda + \sum_{i=1}^n -\eta_i  \geq 0, \alpha_i + \eta_i \geq 0,\text{ for } i \in [n], \boldsymbol{\eta} \geq 0, 
\end{array}
\end{align*}
where the first equivalence follows from the 0/1 extreme points of the unit hypercube and the second equivalence is from linear programming duality.
Constraint \eqref{constr:1_general} can be rewritten as follows,
\begin{align}
\begin{array}{rllll}
 & \displaystyle \min \left\lbrace \sum_{i=1}^n \alpha_i c_i:   \max_{\bold{x} \in \mathcal{X}} \bold{c}^\top \bold{x} \geq r,  \bold{c} \in \{0,1\}^n  \right\rbrace  \geq 1 - \lambda  \nonumber \\
 \iff  & \displaystyle \min  \left\lbrace\sum_{i=1}^n \alpha_i c_i  :  \bold{c}^\top \bold{x} \geq r,  \bold{c} \in \{0,1\}^n \right\rbrace \geq 1 - \lambda,  \text{ for } \bold{x} \in \mathcal{X}\nonumber \\
\iff &\displaystyle  \min  \left\lbrace\sum_{i=1}^n \alpha_i c_i  :  \bold{c}^\top \bold{x} \geq r,  0\leq c_i \leq 1,  \text{ for }  i \in [n] \right\rbrace \geq 1 - \lambda,  \text{ for } \bold{x} \in \mathcal{X}, \nonumber
\end{array}
\end{align}
where the first equivalence is by disaggregating the constraints and the second equivalence follows from the observation that the for each $\bold{x} \in \mathcal{X} \subseteq \{0,1\}^n$, the constraint  $\bold{c}^\top \bold{x} \geq r$  has a totally unimodular structure. Note that while this totally unimodular structure arises with binary support, it breaks down for more general discrete support.
Further dualizing the linear program for each $\bold{x} \in \mathcal{X}$ and enforcing the constraints gives the equivalent reformulation:
\begin{align*}
\begin{array}{rllll}
 &\displaystyle \lambda + r\Delta_{\bold{x}} - \sum_{i=1}^n \gamma_{i,\bold{x}} \geq 1, \text{ for } \bold{x} \in \mathcal{X},\\
&\displaystyle  \alpha_i - \Delta_{\bold{x}}x_i + \gamma_{i,\bold{x}} \geq 0, \text{ for } i \in [n], \bold{x} \in \mathcal{X}, \\
&\displaystyle \Delta_{\bold{x}} \geq 0, \text{ for } \bold{x} \in \mathcal{X},
\gamma_{i,\bold{x}} \geq 0  \text{ for } i \in [n], \bold{x} \in \mathcal{X}.
\end{array}
\end{align*}
 Putting the reformulations together in place of the dual constraints \eqref{constr:0_general} and \eqref{constr:1_general} in the exponential sized dual linear program gives:
\begin{align*}
\begin{array}{rllll}
\min \;\;&\displaystyle \lambda + \sum_{i=1}^n \alpha_i p_i \\
\mbox{s.t.}&\displaystyle \lambda + \sum_{i=1}^n -\eta_i  \geq 0, \\
 &\displaystyle \alpha_i + \eta_i \geq 0,  \text{ for } i \in [n], \\
 &\displaystyle \lambda + r\Delta_{\bold{x}} - \sum_{i=1}^n \gamma_{i,\bold{x}} \geq 1, \text{ for } \bold{x} \in \mathcal{X},\\
&\displaystyle  \alpha_i - \Delta_{\bold{x}}x_i + \gamma_{i,\bold{x}} \geq 0, \text{ for } i \in [n], \bold{x} \in \mathcal{X}, \\
&\displaystyle \Delta_{\bold{x}} \geq 0, \text{ for } \bold{x} \in \mathcal{X},
\gamma_{i,\bold{x}} \geq 0,  \text{ for } i \in [n],  \bold{x} \in \mathcal{X} ,\eta_i \geq 0,  \text{ for } i \in [n].
\end{array}
\end{align*}
Taking the dual of the linear program gives us  the formulation in the theorem.
\end{proof}
This linear program has $O(nP)$ variables and $O(nP)$ constraints. In comparison, the linear program in \Cref{thm:probability_bound_general_discrete} applied to Bernoulli random variables  has $O(n^2P)$ variables and $O(n^2P)$ constraints. Next we describe the construction of the extremal distribution using the optimal solution of the linear program in Theorem \ref{thm:general_binary_support}. Given an optimal solution of the linear program denoted by $\bold{a}^*, b, \boldsymbol{g}^*, \boldsymbol{h}^*$, an extremal distribution is constructed using the following mixture distribution:
\begin{enumerate}
\item Generate a Bernoulli random variable $\tilde{z} = 1$ with probability $\sum_{\bold{x} \in \mathcal{X}} a_{\bold{x}}^*$.
\item If $\tilde{z} = 1$,
\begin{enumerate}
\item Generate $\bold{x} \in \mathcal{X}$ with probability $a_{\bold{x}}^*/\sum_{\bold{x} \in \mathcal{X}} a_{\bold{x}}^*$.
\item For each $i \in [n]$, generate $\tilde{c}_i = 1$ with probability $g_{i,\bold{x}}/a_{\bold{x}}^*$ and $\tilde{c}_i = 0$ otherwise..
\end{enumerate}
\item If $\tilde{z} = 0$, for $i \in [n]$, generate $\tilde{c}_i = 1$ with probability $h_i^*/b$.
\end{enumerate}

\subsubsection{Weighted Probability Bounds}\label{subsec:weighted}
 In this section, we show that the results in \Cref{cor:sums_discrete} can be extended to compute tight weighted probability bounds of sums of discrete random variables as the optimal value of a compact linear program. Such bounds are useful in modeling scenarios where some of the variables are extremally dependent (assuming only knowledge of the marginal distributions), while the rest are mutually independent and the two sets of variables are independent of each other (see \Cref{subsec:limiteddependence} for a numerical example). We can thus offset the inherent conservatism in the extremally dependent and mutually independent models by introducing  a limited degree of independence into the model. Denote by $\bold{w}=(w_1,w_2, \dots, w_{nK}), \; w_i \in \mathbb{R}, \;  i \in [nK]$ a vector of pre-specified weights. We are interested in computing the following tight upper bound on the weighted sum of the tail probabilities
$$\underset{\theta \in \Theta}{\max}\;\sum_{l=0}^{nK} w_l \mathbb{P}_{\theta}\left(\sum_{i=1}^n \tilde{c}_i \geq  l\right).$$
Note that without loss of generality, we can ignore $\ell=0$ and consider $\mathbb{P}_{\theta}(\sum_{i=1}^n \tilde{c}_i =  l)$ for $\ell \in [nK]$ instead of tail probabilities by a suitable transformation of weights.
Denote by $S(\bold{w},K)$ the following upper bound:
\begin{align*}
S(\bold{w},K) = \underset{\theta \in \Theta}{\max}\;\sum_{l=1}^{nK} w_l \mathbb{P}_{\theta}\left(\sum_{i=1}^n \tilde{c}_i =  l\right),
\end{align*}
where we are given the marginal distributions of the discrete random vector $\tilde{\bold{c}}$ as $\mathbb{P}(\tilde{c}_i = k) = p_{ik}$ for $k \in [0,K]$ and $i \in [n]$.
We next prove the result for sums of Bernoulli random variables ($K=1$) which can then be extended to sums of discrete variables with $K \geq 2$.

\begin{theorem} \label{thm:weightedLPBernoulli}
Given the marginal distributions of a Bernoulli random vector $\tilde{\bold{c}}$ as $\mathbb{P}(\tilde{c}_i = 1) = 1-\mathbb{P}(\tilde{c}_i = 0)=p_{i}$ for $i \in [n]$, the tightest upper bound $S(\bold{w},1)$  is computable by solving the linear program:
\begin{equation}\label{opt:primalsumweightcompact}
\begin{array}{llll}
S(\bold{w},1) =&\max & \displaystyle\sum_{l=0}^n \tau_{l}w_{l} \\
&\mbox{s.t.} & \displaystyle\sum_{l=0}^n  \tau_{l}=1,&\\
&&\displaystyle\sum_{l=0}^n  \delta_{li}=p_{i},& \;\mbox{for}\; i \in [n], \\
&&\tau_{l}\geq \delta_{li},& \;\mbox{for}\; i \in [n],\;\mbox{for}\; l \in [n],\\
&&\displaystyle\sum_{i=1}^n   \delta_{li}= l\tau_{l},& \;\mbox{for}\; l \in [n], \\
&&\tau_{l}\geq 0,&\;\mbox{for}\; l \in [n],\\
&&\delta_{li}\geq 0, &\;\mbox{for}\; i \in [n],\;\mbox{for}\; l \in [n].\\
\end{array}
\end{equation}
\end{theorem}

\begin{proof}
The tight bound $S(\bold{w},1)$ can be computed as the optimal value of the following exponential sized linear program:
\begin{equation}
\begin{array}{llllll} \label{opt:exposumweight}
&\max &  \displaystyle\sum_{l=0}^n w_l  \underset{\bold{c} \in [0,1]^n:\sum_{t=1}^n c_t = \ell}{\displaystyle \sum}\theta(\bold{c})\\
&\textrm{s.t.} &  \underset {\scriptstyle {\bold{c}\in [0,1]^n:c_i =  1}}{\sum} \theta(\bold{c}) = p_{i} , & \;\mbox{for}\; i \in [n],\\
 & &\quad \underset{\bold{c}\in [0,1]^n}{\sum} \theta(\bold{c})  = 1, \\
 & &  \quad\theta(\bold{c}) \geq 0 & \;\mbox{for}\; \bold{c} \in [0,1]^n.
\end{array}
\end{equation}
An optimal solution of this linear program always exists with a finite optimal value. Note that when $\bold{w}= (\bold{0}_{r-1}, \bold{1}_{n-{r+1}})$ (zeros up to index $r-1$ and ones thereafter), the objective function in \eqref{opt:exposumweight} reduces to the tail probability bounds $S(r,1)$ considered in \Cref{subsec:discretesums}. We next derive a compact reformulation of \eqref{opt:exposumweight} by considering the linear relaxation of its dual separation problem, similar to the proof of \Cref{thm:general_binary_support} with $\mathcal{X}=\{\bold{1}_n\}$. The dual of the linear program \eqref{opt:exposumweight} can be written as:
\begin{equation}
\begin{array}{llll} \label{opt:dualsumweight}
&\min & \displaystyle\sum_{i=1}^n  \alpha_{i}p_{i}+\lambda &\\
&\textrm{s.t.} &\displaystyle\sum_{i=1}^n  \alpha_{i} c_i +\lambda\geq w_{l},& \;\mbox{for}\; \bold{c}\in [0,1]^n: \sum_{i=1}^n c_i =l,\; \;\mbox{for}\; l \in [n].\\
\end{array}
\end{equation}

\noindent The dual linear program \eqref{opt:dualsumweight} has $2^n$ constraints, which can be divided into $n$ sets of $\binom{n}{l}$ constraints for $l \in [n]$. Similar to the steps followed in the derivation of the reduced formulation for Bernoulli variables in \Cref{thm:general_binary_support}, for each $l \in [n]$, the set of $\binom{n}{l}$ constraints corresponding to the scenarios  $\bold{c}\in [0,1]^n: \sum c_i =l$  can be rewritten as follows:
\begin{equation}
\begin{array}{lll}  \label{opt:dualsumweightrelax}
& \lambda+\left\{\min \displaystyle\sum_{i=1}^n  \alpha_{i} c_i: \bold{c}\in  [0,1]^n,\; \sum_{i=1}^n {c}_i = l\right\} \geq w_l, & \;\mbox{for}\; l \in [n] \\
 \iff & \lambda+\left\{\min
\displaystyle\sum_{i=1}^n  \alpha_{i}{c}_i :
 0\leq c_i \leq 1,\;\mbox{for}\; i \in [n],\;
 \displaystyle \sum_{i=1}^n c_i =l
\right\}  \geq w_l, & \;\mbox{for}\; l \in [n] \vspace{0.4cm}\\
 \iff &\lambda+\left\{\displaystyle
\begin{array}{lll}
\max & \displaystyle\sum_{i=1}^n  u_{li}+lv_{l} \\
\textrm{s.t.} &u_{li}+v_{l}\leq\alpha_{i}, &\;\mbox{for}\; i \in [n], \\
&  u_{li}\leq 0,&\;\mbox{for}\; i \in [n],\\
\end{array}\right\}\geq w_l,  &\;\mbox{for}\; l \in [n],
\end{array}
\end{equation}
where the first equivalence follows from the totally unimodular structure of the constraint matrix and the second equivalence is from linear programming duality.
Since an optimal solution to the primal \eqref{opt:exposumweight} exists, by strong duality, the dual \eqref{opt:dualsumweight} must also have an optimal solution. Consequently there must exist a feasible solution to the linear program in the last equivalence of \eqref{opt:dualsumweightrelax} and the constraint sets corresponding to each $l \in [0,n]$  in \eqref{opt:dualsumweight} can be replaced by the following polynomial-sized set of constraints:

\begin{equation}
\left\{
\begin{array}{lll}  \label{dweight5}
&\lambda+{\displaystyle\sum_{i=1}^n  u_{li}+lv_{l} \geq w_l}, &\\
&u_{li}+v_{l}\leq\alpha_{i}, &\;\mbox{for}\; i \in [n],\\
 &u_{li}\leq 0,&\;\mbox{for}\; i \in [n], \\
\end{array}\right\} \text{\parbox{3.5cm}{,\qquad $\;\mbox{for}\; l \in [n].$}}
\end{equation}

\noindent Thus the compact version of the dual \eqref{opt:dualsumweight} can be written as:
\begin{equation}
\begin{array}{llll}  \label{opt:dualsumweightcompact}
&\min & \displaystyle\sum_{i=1}^n  \alpha_{i}p_{i}+\lambda &\\
&\textrm{s.t.} & \lambda-{\displaystyle\sum_{i=1}^n  u_{li} +l	v_{l} \geq w_l},  &\;\mbox{for}\;  l\in [n], \\
&&v_{l}-u_{li}\leq\alpha_{i},& \;\mbox{for}\; i \in [n], \;\mbox{for}\; l \in [n], \\
&&u_{li}\geq 0, & \;\mbox{for}\; i \in [n],\;\mbox{for}\; l \in [n]. \\
\end{array}
\end{equation}
\noindent Finally, dualizing \eqref{opt:dualsumweightcompact} leads to the compact linear program \eqref{opt:primalsumweightcompact} with $O(n^2)$ variables and constraints.
\end{proof}

It is straightforward to generalize the result in \Cref{thm:weightedLPBernoulli} to compute the tight bound on the weighted probability of sums of discrete random variables $S(\bold{w},K)$ by a combination of techniques used in the proofs of \Cref{cor:sums_discrete} and \Cref{thm:weightedLPBernoulli}.
\subsection{Hardness Results for the Lower Bound and Independence}
In this section, we show both $L(r)$ and $I(r)$ are not computable in polynomial time for compact 0/1 V-polytopes unless P = NP. The hardness results are shown using a reduction from the independent set problem in graphs. An independent set in an undirected graph $G = (V,E)$ is a subset of the vertices such that no two vertices are adjacent to one another. The decision and optimization version of this problem are known to be NP-hard while counting the number of independent sets is known to be \#P hard \cite{garey1979computers}. The next theorem shows computing the lower bound $L(r)$ is NP-hard.
\begin{theorem}
\label{thm:hardness_extremal_v}
Let ${\cal X} = \{\bold{x}^1, \ldots, \bold{x}^P \} \subseteq \{0,1\}^n$. Given the marginal distributions of the Bernoulli random vector $\tilde{\bold{c}}$ as $\mathbb{P}(\tilde{c}_i = 1) = 1-\mathbb{P}(\tilde{c}_i = 0) = p_{i}$ for $i \in [n]$, computation of the lower bound $L(r)$ is NP-hard and cannot be computed in time polynomial in the input size unless P = NP.
\end{theorem}
\begin{proof}
The dual linear program for computing $L(r)$ is given by:
\begin{align*}
L(r) = \max & \; \; \displaystyle\lambda + \sum_{i=1}^n \alpha_{i} p_{i} \nonumber\\
\mbox{s.t.}&\; \;  \displaystyle\lambda + \sum_{i=1}^n \alpha_{i} c_i \leq 1, \text{ for } \bold{c} \in [0,1]^n, \\
& \; \; \displaystyle \lambda + \sum_{i=1}^n \alpha_{i} c_i  \leq 0, \text{ for } Z(\bold{c}) \leq r-1, \bold{c} \in  [0,1]^n,
\end{align*}
where the decision variables are $\lambda$ and $\alpha_{i}$ for $i \in [n]$.
The relevant separation problem to be solved to compute $L(r)$ boils down to:
\begin{align}
\label{constr:dual_separation_Z(c)_less_than_r}
\max \left\lbrace \sum_{i=1}^n \alpha_i c_i: \bold{c}'\bold{x}^j \leq r-1, \text{ for } j \in [P], c_i \in \{0,1\}, \text{ for } i \in [n]  \right \rbrace,
\end{align}
where $\boldsymbol{\alpha} \in \mathbb{R}^n$  is given. This is NP-hard to solve.
To see this, consider a graph $G= (V,E)$ on $n$ nodes. Given an undirected graph $G = (V,E)$, let $n = |V|$ and $P = |E|$. Define the set ${\cal X}$ as the set of incidence vectors of the graph:
$$\displaystyle {\cal X} = \{\bold{x}^{e}; e \in E\} \subseteq \{0,1\}^n,$$ where for any $e = (i,j) \in E$, we let $x^{e}_{i} = 1$, $x^e_{j} = 1$ and $x^{e}_k = 0$ for all $k \neq i, j$. Setting $\alpha_i = 1$ for all $i$ and  $r=2$ in \eqref{constr:dual_separation_Z(c)_less_than_r} solves the maximum independent set problem. Since the separation problem is NP-hard to solve, the optimization problem is NP-hard to solve and computing $L(r)$ is NP-hard.
\end{proof}

We next discuss hardness results for computing the probabilities with independent random variables. The next theorem is taken from \cite{kleinberg} who showed that computing the probability of the sum of independent discrete random variables is \#P-hard.
\begin{theorem} \label{thm:knap} \cite{kleinberg}
Let $\tilde{c}_i$ be a two point random variable with $\mathbb{P}(\tilde{c}_{i} =a_i) = 1-\mathbb{P}(\tilde{c}_{i} = 0)=p_{i}$ for $a_i \in \mathbb{Z}_{+}$. Computing the probability $I(r) = \mathbb{P}_{\theta_{ind}}(\sum_{i=1}^{n}\tilde{c}_i \geq r)$ is \#P-hard.
\end{theorem}
The hardness in Theorem \ref{thm:knap} was shown using a reduction from the counting version of the knapsack problem. The hardness result in their construction arises from the support of the random variables. Specifically when the random variables have restricted support such as Bernoulli, the sum is a Poisson Binomial random variable for which the probability is computable in polynomial time through recursion \cite{chen1997statistical}. We next show however that for $Z(\tilde{\bold{c}})$ given as the optimal value of a maximization problem over a compact 0/1 V-polytope, computing the probability under the assumption of independence is hard even when the random variables are Bernoulli.

\begin{theorem}
\label{thm:hardness_independence}
Let ${\cal X} = \{\bold{x}^1, \ldots, \bold{x}^P \} \subseteq \{0,1\}^n$. Given the marginal distributions of the Bernoulli random vector $\tilde{\bold{c}}$ as $\mathbb{P}(\tilde{c}_i = 1) = 1-\mathbb{P}(\tilde{c}_i = 0) = p_{i}$ for $i \in [n]$, computation of the probability $\mathbb{P}_{\theta_{ind}} (Z(\tilde{\bold{c}}) \geq r)$ is \#P-hard and cannot be computed in time polynomial in the input size unless P = NP.
\end{theorem}

\begin{proof}
We will do a reduction from counting the number of independent sets in a graph.
Given an undirected graph $G = (V,E)$, let $n = |V|$ and $P = |E|$. Define the set ${\cal X}$ as the set of incidence vectors of the graph:
$$\displaystyle {\cal X} = \{\bold{x}^{e}; e \in E\} \subseteq \{0,1\}^n,$$ where for any $e = (i,j) \in E$, we let $x^{e}_{i} = 1$, $x^e_{j} = 1$ and $x^{e}_k = 0$ for all $k \neq i, j$. Let $\mathbb{P}(\tilde{c}_i = 1) = 1-\mathbb{P}(\tilde{c}_i = 0) = 1/2$ and $r = 2$. Then:
\begin{align*}
\mathbb{P}_{\theta_{ind}} \left(\max_{e \in E} \tilde{\bold{c}}'\bold{x}^{e} \geq 2 \right) &= 1-\mathbb{P}_{\theta_{ind}} \left(\max_{e \in E} \tilde{\bold{c}}'\bold{x}^{e} \leq 1 \right) \\
 &= 1-\mathbb{P}_{\theta_{ind}} (\tilde{c}_i + \tilde{c}_j \leq 1 \text{ for } (i,j) \in E)  \\
&= 1-\mathbb{P}_{\theta_{ind}} (\tilde{\bold{c}} \text{ induces an independent set on G}) \\
& = 1-\frac{\text{No. of independent sets in G}}{2^n}.
\end{align*}
Since computing the number of independent sets is \#P-hard, so is computing $I(r)$.
\end{proof}

\section{Bounds for the H-Polytope: PERT Networks}
\label{sec:h_polytope}
In this section, we consider combinatorial optimization problems with a known compact H-polytope representation. While the formulations in the previous section can be used for V-polytope representations, the complexity of the formulations depend on $P$ and can be cumbersome in applications where $P$ is large.  It is therefore desirable to have compact formulations under known H-polytope representations.  We will now show that for PERT networks represented with a H-polytope, the upper bound $U(r)$ is efficiently computable in polynomial time in $n$ and $K$.
PERT networks are widely used  in project planning and management across various settings such as construction projects, software planning projects and facility maintenance projects.  A PERT network is denoted by a directed acyclic graph (DAG) $G = (V, E)$ where $V$ is the set of vertices and $E$ is the set of edges. The start node is denoted by $s \in V$ and the terminal node is denoted by $t \in V$.  The arcs represent activities in the project and nodes represent events in an activity on arc framework \cite{elmaghraby1977activity}. The network structure captures precedence relationships among the activities. Each activity is associated with a random time duration to complete that activity. For fixed activity durations denoted by $c_{ij}$ for $(i,j) \in E$, the completion time of the project is computed as the longest path from node $s$ to $t$. This is formulated as the 0-1 integer program:
\begin{align*}
\begin{array}{rllllll}
  Z^{\text{pert}}(\bold{c}) =  \max  & \displaystyle \sum_{(i,j) \in E} c_{ij} x_{ij}  \\
   \mbox{s.t} & \displaystyle \sum_{j: (i,j) \in E} x_{ij} - \sum_{j:
    (j,i) \in E} x_{ji} =
\begin{cases}
1, \quad &\mbox{if } i=s, \\
 -1,  &\mbox{if } i=t, \\
0, & \mbox{otherwise, } \\
\end{cases} \\
&  x_{ij} \in \{0,1\},  \quad\text{ for }  (i,j) \in E.
\end{array}
\end{align*}
The total unimodularity of the constraint matrix ensures that the  LP relaxation exactly solves the integer program and $Z^{pert}(\bold{c})$ is polynomial time computable.

There is a large stream of literature on uncertain PERT networks \cite{wiesemann2012optimization,ROOS2021918} and computing the distribution and the expected value of $Z^{\text{pert}}(\tilde{\bold{c}})$ with independent activity durations. Evaluating both the distribution and the expected value are known to be \#P-hard \cite{Hagstrom1988} and not polynomial time computable even in the number of values that the project duration takes. Several approximations and bounds have been proposed (see  \cite{Fulkerson1962,Dodin1985,Kleindorfer1971}). In special cases, the computation of the distribution and the expected value are known to be possible in polynomial time with independent distributions.  Specifically, for the class of series parallel graphs with activity durations supported in $[0,K]$, the worst case probability and expectation bounds can be computed in polynomial time. For more general graphs, prior works of \cite{Dodin1985,Kleindorfer1971} have also constructed approximations by using  transformations to series parallel graphs.

Applying the formulation in \Cref{thm:probability_bound_general_discrete} requires enumeration of the $P$ extreme points which in the setting of PERT networks, corresponds to the $s$-$t$ paths in the network. The previous formulation is hence useful only when the number of $s$-$t$ paths does not grow rapidly.
We next propose a tight formulation that does not require the enumeration of the $s$-$t$ paths. Specifically the result implies that for extremal dependence, the worst-case probability is polynomial time computable for general DAG under the assumption of restricted support in $[0,K]$ while for independent distributions, such a result is possible only for restricted graphs like series parallel graphs.

\begin{theorem}
\label{thm:pert_primal_lp}
Consider a PERT network $G = (V,E)$ with $|E| =n$ and $s$ and $t$ denoting the source and terminal nodes respectively. Given the marginal distributions of the activity duration vector $\tilde{\bold{c}}$ as $\mathbb{P}(\tilde{c}_{ij} = k) = p_{ijk}$ for $(i,j) \in E$, $k \in  [0,K]$ and $r \in [0, nK]$, the tightest upper bound on the probability of the project completion time taking a value greater than or equal to $r$ is the optimal value of the linear program:
\begin{align}
U^{pert}(r) = \max &\;\; a \nonumber\\
\mbox{s.t.}&\;\; a + b = 1, \label{eq:pert_a+b}\\
& \sum_{k =0}^K h_{ijk} = b,  \text{ for } (i,j) \in E,   \label{eq:pert_h}\\
&  \sum_{k =0}^K g_{ijk} + \sum_{k =0}^K \sum_{l=k}^{nK} \delta_{ij,k,l} = a, \text{ for } (i,j) \in E,  \label{eq:g+delta=a} \\
& h_{ijk} + g_{ijk} + \sum_{l=k}^{nK} \delta_{ij,k,l} = p_{ijk}, \text{ for } (i,j) \in E, k \in  [0,K],\label{eq:consistency_with_marginals} \\
& a = \sum_{l=r}^{nK} \tau_l, \label{eq:a_tau} \\
& \tau_l = \sum_{i: (i,t) \in E} \sum_{k =0}^{\min(l,K)} \delta_{it, k, m}, \text{ for } l \in [r, nK] ,\label{eq:tau_m} \\
&\sum_{j:(i,j) \in E} \sum_{k = 0}^{\min(K, nK-l)} \delta_{ij, k, l+k} = \sum_{j:(j,i) \in E}\sum_{k = 0}^{\min(K,l)} \delta_{ji, k, l},  \label{eq:delta_any_node}\\
& \;\;\;\;\;\;\;\;\;\text{ for } i \in V \setminus \{s, t \}, l \in [0, nK], \nonumber \\
& \sum_{i: (i,t) \in E}  \sum_{k =0}^{\min(l,K)} \delta_{it, k, l}  = 0, \text{ for } l \in [0, r-1] ,\label{eq:delta_dest} \\
& \sum_{i:(s,i) \in E} \sum_{k = 0}^{\min(K, nK-l)} \delta_{si, k, l+k}  = 0,  \text{ for } l \in [1,nK] ,\label{eq:delta_source} \\
& a, b \geq 0, \tau_l \geq 0, \text{ for } l \in [r, nK] ,\nonumber \\
&h_{ijk}, g_{ijk} \geq 0, \text{ for } (i,j) \in E, k \in [0,K] ,\nonumber \\
& \delta_{ij,k,l} \geq 0, \text{ for } (i,j) \in E, k \in  [0,K], l \in [k, nK] .\nonumber
\end{align}
\end{theorem}
Namely $\max_{\theta \in \Theta} \mathbb{P}_{\theta}(Z^{pert}(\tilde{\bold{c}}) \geq r) = U^{pert}(r)$.
\begin{proof}
The approach will, as before, involve developing a compact formulation for the separation problem in (\ref{constr:1_general}). We will make use of the structure of the $s$-$t$ flow polytope in order to derive the reduced formulation. Given $\lambda$ and $\boldsymbol{\alpha}$, the constraint (\ref{constr:1_general}) is equivalent to:
\begin{align}
\label{sep:min_cost_assignment}
  \lambda + \underbrace{\min \left\lbrace \sum_{(i,j) \in E} \sum_{k =0}^K \alpha_{ijk} \mathbbm{1}_{\{c_{ij} = k\}}:  Z^{pert}(\bold{c}) \geq r, c_{ij} \in  [0,K], \text{ for } (i,j) \in E  \right\rbrace}_{\mbox{Sep}(\boldsymbol{\alpha})} \geq 1.
\end{align}
This problem looks at assigning a length from the set $ [0,K]$ to each edge $c_{ij}$ where  the cost of assigning  length $k$ to $c_{ij}$ is $\alpha_{ijk}$. In particular, we want to compute a minimum cost assignment of the lengths to $c_{ij}$ in such a way that the longest path from node $s$ to $t$ has a length at least $r$. This is equivalent to ensuring the existence of a $s$-$t$ path with length at least $r$. The costs $\boldsymbol{\alpha}_{ij} = (\alpha_{ijk}; k \in [0,K])$ can be viewed as a mapping from $ [0,K]$ to $\mathbb{R}$, albeit without any  structural assumptions such as monotonicity, non-negativity etc.
 Observe that for each edge $(i,j) \in E$, we will always incur a cost of at least $q_{ij} = \min_{k \in  [0,K]} \alpha_{ijk}$.   We focus on minimizing the updated costs $v_{ijk} = \alpha_{ijk} - q_{ij} \geq 0$. In particular for $k^* \in \argmin_{k \in  [0,K]} \alpha_{ijk}$, we have $v_{ijk^*}=0$. The optimization problem $\mbox{Sep}(\boldsymbol{\alpha})$ in (\ref{sep:min_cost_assignment}) can therefore be split up as follows:
 \begin{align}
 \label{sep:pert_opt_problem}
\displaystyle \mbox{Sep}(\boldsymbol{\alpha}) = \sum_{(i,j) \in E} q_{ij} +  \left\{ \begin{array}{ll}
& \displaystyle \min \;\; \sum_{(i,j) \in E} \sum_{k =0}^K \overbrace{(\alpha_{ijk} - q_{ij})}^{v_{ijk}} \mathbbm{1}_{\{c_{ij} = k\}} \\
 &\displaystyle \,\mbox{s.t} \;\; Z^{pert}(\bold{c}) \geq r, c_{ij} \in  [0,K] \text{ for } (i,j) \in E
 \end{array} \right\}.
  \end{align}
 We will now focus on finding an assignment to $\bold{c}$ so as to solve the optimization problem in the second term in \Cref{sep:pert_opt_problem}. Observe that we want to minimize the updated costs $\bold{v}$ subject to the constraint  $Z(\bold{c}) \geq r$. For this, we propose a set of dynamic programming recursions as follows.

 Let $f_{l,i}$ denote the best value of the objective in the optimization problem \eqref{sep:pert_opt_problem} when there exists a path from $s$ to $i$ with of length exactly $l$.  The computation of $f_{l,i}$ gives a minimum cost assignment such that some path from $s$ to $i$ has a length of exactly $l$. Since a PERT network is described by a DAG, there exists an ordering of the vertices by means of a topological sort. Denote such an ordering by $O_{top}$. The base case of the dynamic program is given by the  computation of $f_{0,s}$ for the source node $s$. Clearly $f_{0,s} = 0$ as the assignment $c_{ij} = \argmin_{k \in S} v_{ijk}$ incurs a total cost of $0$ and any path from $s$ to itself has a length of $0$ trivially.
Next we describe the induction step.  For any node $j$, let the value of $f_{l,i}$ be known for all nodes $i$ such that $(i,j) \in E$, $l \in [0,  nK]$.  This is possible when we fill the columns of the matrix $f$  in the order given by $O_{top}$. The following relations hold,
\begin{align*}
\displaystyle f_{l,j} = \min_{i: (i,j) \in E} \min_{k \in  [0,K]} \left(f_{l-k, i} + v_{ijk} \right), \text{ for } l \in  [k,  nK]
\end{align*}
This hold since if a path of length $l$  exists from $s$ to $j$ and an edge $(i,j)$ on this path  is assigned a value of $k$, then the path from $s$ to $i$ must have a length of $l-k$. The optimal value of the objective must therefore choose the minimum value generated out of all possible assignments for all incoming arcs $(i,j)$ to node $j$. Finally the objective function in \eqref{sep:pert_opt_problem} requires that the assignment produces a path of length of at least $r$ from $s$ to $t$. Let $z$ denote the objective value of the optimization problem in \eqref{sep:pert_opt_problem}.  Then,
$z = \min_{l \in [r, nK]} f_{l,t}.$ Putting all the dynamic programming recursions together gives us the following compact linear program for $\mbox{Sep}(\boldsymbol{\alpha})$:
\begin{align*}
\begin{array}{rrlll}
\displaystyle \mbox{Sep}(\boldsymbol{\alpha})= & \max\limits_{\bold{q}, \bold{f}, z}&\displaystyle \sum_{(i,j) \in E} q_{ij} + z \\
&\mbox{s.t.}&\displaystyle q_{ij} \leq \alpha_{ijk}, \text{ for } (i,j) \in E, k \in  [0,K], \\
&&\displaystyle f_{0,s}  = 0,\\
&&\displaystyle f_{l, j} \leq f_{l-k, i} + \alpha_{ijk} - q_{ij}, \text{ for } (i,j) \in E, k \in [0,K], l \in [k, nK] , \\
&&\displaystyle z \leq f_{l, t}, \text{ for } l \in [r, nK].
\end{array}
\end{align*}
Now, forcing the above linear program to take a value greater than 1 gives the following reformulation for \eqref{constr:1_general} in the exponential sized dual formulation,
 \begin{align*}
\begin{array}{rlll}
& \lambda + \sum_{(i,j) \in E} q_{ij} + z \geq 1, \\
& q_{ij} \leq \alpha_{ijk}, \text{ for } (i,j) \in E, k \in  [0,K], \\
& f_{0,s}  = 0,\\
& f_{l, j} \leq f_{l-k, i} + \alpha_{ijk} - q_{ij}, \text{ for } (i,j) \in E, k \in [0,K], l \in [k, nK] , \\
& z \leq f_{l, t}, \text{ for } l \in [r, nK].
\end{array}
\end{align*}
 Constraint \eqref{constr:0_general} can be reformulated in the same manner as described in proof of \Cref{thm:probability_bound_general_discrete}. Combining the  reformulations for  \eqref{constr:1_general} and \eqref{constr:0_general} gives us,
\begin{align*}
\begin{array}{rlll}
\min & \displaystyle \lambda + \sum_{(i,j) \in E} \sum_{k \in [0,K]} \alpha_{ijk} p_{ijk} \\
\mbox{s.t.}& \displaystyle \lambda + \sum_{(i,j) \in E} d_{ij} \geq 0,\\
 &\displaystyle  \alpha_{ijk} - d_{ij} \geq 0, \text{ for } (i,j) \in E, \text{ for } k \in [0,K], \\
&\displaystyle  \lambda + z+ \sum_{(i,j) \in E} q_{ij}  \geq 1, \\
&\displaystyle  \alpha_{ijk}- q_{ij} \geq 0, \text{ for } (i,j) \in E, k \in  [0,K], \\
&\displaystyle  f_{0,s}  = 0,\\
&\displaystyle  f_{l-k, i} + \alpha_{ijk} - q_{ij} - f_{l,j} \geq 0, \text{ for } (i,j) \in E, k \in [0,K], l \in [k, nK] , \\
&\displaystyle  f_{l, t}- z \geq 0, \text{ for } l \in [r, nK].
\end{array}
\end{align*}
 Further taking the dual of this linear program gives us the formulation in the theorem.
\end{proof}
The techniques used in deriving \Cref{thm:probability_bound_general_discrete} and \Cref{thm:pert_primal_lp} rely on dynamic programming. However by making further use of the problem structure, we are able to obtain a further reduced formulation in \Cref{thm:pert_primal_lp} for PERT networks.

\section{Numerical Results}
\label{sec:numericals}
In this section, we provide numerical results from different formulations. All computations were carried out using Gurobi \cite{gurobi} on a Python interface. The following probabilities are computed in different examples.
\texitem{(a)} Upper bound $U(r)$: The tightest upper bounds are computed using the linear programs in Theorem \ref{thm:probability_bound_general_discrete} and \ref{thm:pert_primal_lp}.
\texitem{(b)} Markov bound: Using Markov's inequality gives us a valid upper bound for any distribution $\theta \in \Theta$ and positive value of $r$:
\begin{align*}
\displaystyle \mathbb{P}_{\theta}(Z(\tilde{\bold{c}}) \geq r) \leq \min\left(\max_{\theta \in \Theta} \mathbb{E}_{\theta}[Z(\tilde{\bold{c}})]  / r,1\right).
\end{align*}
To compute the maximum expected value when $\mathcal{X} $ is represented with a V-polytope, we can use existing results in the literature. Specifically using the formulation proposed in \cite{meilijson1979convex}, we get:
\begin{align*}
\begin{array}{rllll}
\displaystyle \max_{\theta \in \Theta} \mathbb{E}_{\theta}[Z(\tilde{\bold{c}})] = \max_{\boldsymbol{\gamma}, \boldsymbol{\lambda}} &\displaystyle\sum_{i=1}^n \sum_{k=0}^K k \gamma_{ik} p_{ik} \\
\mbox{s.t.}&\displaystyle \sum_{\bold{x} \in \mathcal{X}} \lambda_{\bold{x}} = 1,\\
&\displaystyle \sum_{k =0}^{K}p_{ik} \gamma_{ik}  = \sum_{\bold{x} \in \mathcal{X}: x_i = 1} \lambda_{\bold{x}},
 \\
&\displaystyle 0 \leq \gamma_{ik} \leq 1 \; \text{ for } i \in [n], \\
&\displaystyle \lambda_{\bold{x}} \geq 0 \, \text{ for } \bold{x} \in \mathcal{X},
\end{array}
\end{align*}
where the random variables $\tilde{c}_i$ take support in $[0,K]$ with $p_{ik} = \mathbb{P}(\tilde{c}_i = k)$ for $k \in [0,K]$ and $i \in [n]$. 
\texitem{(c)} Independence: To compute $I(r) = \mathbb{P}_{\theta_{ind}}(Z(\tilde{\bold{c}} \geq r)$, we approximate the probability using a simulation of 10000 runs.
\texitem{(d)} Distribution maximizing $\mathbb{E}[Z(\tilde{\bold{c}} - r) ^+]$: A formulation that computes this maximum expectation can be derived using the techniques in \cite{meilijson1979convex,Doan2012}. We provide the formulation below.
\begin{align}
\label{opt:max_exp_z_c_minus_r}
\begin{array}{rllll}
 \displaystyle \max \mathbb{E}[Z(\tilde{\bold{c}})-r]^+  =  \max & \displaystyle \sum_{(i,j) \in E} \sum_{k=0}^K k g_{ijk}p_{ijk} - r \sum_{\bold{x} \in \mathcal{X}} \lambda_{\bold{x}} \\
 \mbox{s.t.} &\displaystyle  \sum_{\bold{x} \in \mathcal{X}} \lambda_{\bold{x}} \leq 1, \\
 &\displaystyle  g_{ijk} \leq 1, \text{ for } (i,j) \in E, k \in [0, K], \\
 &\displaystyle  \sum_{\bold{x} \in \mathcal{X}: x_{ij} = 1} \lambda_{\bold{x}} = \sum_{k =0}^K g_{ijk} p_{ijk}, \text{ for } (i,j) \in E ,\\
 &\displaystyle  g_{ijk} \geq 0, \text{ for } (i,j) \in E, k \in [0, K],\\
 &\displaystyle  \lambda_{\bold{x}} \geq 0, \text{ for } \bold{x} \in \mathcal{X}.
 \end{array}
\end{align}
Extending the results in \cite{Weiss1986} to other applications, the term $\sum_{\bold{x} \in \mathcal{X}} \lambda_{\bold{x}}$  gives us  $\mathbb{P}(Z(\tilde{\bold{c}}) \geq r) $ for the extremal distribution which maximizes $\mathbb{E}[Z(\tilde{\bold{c}})-r]^+$. We refer to this probability bound as `Worst Exp' in all the plots.

\subsection{Sums of Random Variables with Limited Dependence}\label{subsec:limiteddependence}
We first provide a numerical application of the weighted probability bounds to the sums of random variables by allowing for a limited degree of dependence. This is achieved by considering a split of the set of random variables into two sets - one set which allows for extremal dependence among the variables while the other set which contains mutually independent variables. The random variables across the two sets are assumed independent of each other. Specifically let $P(\tilde \alpha_i =1)=1-P(\tilde \alpha_i =0)=p_{i}$ for $i \in [n_1]$ and $P(\tilde \beta_j =1)=1-P(\tilde \beta_j =0)=q_{j}$ for $j \in [n_2]$. The dependence among random variables in $\tilde{\mb{\alpha}}$ are not specified while the random variables in $\tilde{\mb{\beta}}$ are mutually independent. The two sets of random variables are also independent of each other.  Under this model, we will see that the bound on the tail probability of the sum of random variables can be reformulated using the weighted probability bound in Theorem \ref{thm:weightedLPBernoulli} where the weights are appropriately computed.

 Given $ r\in [0,n_1+n_2]$, let the tightest upper bound on the tail probability be given as:
\begin{align}\label{eq:limitedtailtight}
\overline{S}(r,1)  &=\max_{\theta \in \Theta_{\ell}} \mathbb{P}_{\theta}\bigg(\displaystyle \sum_{i=1}^{n_{1}} \tilde \alpha_{i}+\sum_{j=1}^{n_{2}} \tilde\beta_{j}\geq r\bigg),
\end{align}
where $\Theta_{\ell}$ is the set of distributions consistent with the given assumptions:
\begin{align*}
\Theta_{\ell}= \big\{
 \theta \in \mathbb{P}(\{0,1\}^{n_1+n_2}) :\; & \mathbb{P}_\theta\left(\bm{\alpha},\bm{\beta}\right) =\mathbb{P}_\theta\left(\bm{\alpha}\right)\mathbb{P}_{\theta_{ind}}\left(\bm{\beta}\right),
   & \;\mbox{for}\; \left(\bm{\alpha},\bm{\beta}\right)  \in \{0,1\}^{n_1+n_2},\\
    &\mathbb{P}_\theta\left(\tilde{\alpha}_i = 1
    \right) = p_{i} ,
   & \hspace{-1.5cm}\;\mbox{for}\; i \in [n_{1}] \big\},
\end{align*}
where $\theta_{ind}$ is the product distribution for the independent variables in $\tilde{\mb{\beta}}$ supported on $\{0,1\}^{n_2}$. We refer to this as the ``limited dependency'' model. The probability can be rewritten as:
 \begin{equation*}\label{eq:splitobjfnfeas}
\begin{array}{lll}
\mathbb{P}_{\theta}\bigg(\displaystyle \sum_{i=1}^{n_{1}} \tilde \alpha_{i}+\sum_{j=1}^{n_{2}} \tilde\beta_{j}\geq r\bigg)
&=& \displaystyle\sum_{\ell=0}^{n_{2}}\bigg[\mathbb{P}_{\theta_{\alpha}}\big(\displaystyle {\sum_{i=1}^{n_{1}} \tilde{\alpha}_{i}\geq r-\ell\big) \mathbb{P}_{\theta_{ind}}\big(\sum_{j=1}^{n_{2}} \tilde{\beta}_{j}= \ell}\big)\bigg],
\end{array}
\end{equation*}
where $\theta_{\alpha}$ is any feasible distribution of the random vector $\tilde{\mb{\alpha}}$ consistent with the given marginal information and:
 \begin{equation*}
\begin{array}{lll} \label{thetaalpha}
\displaystyle \Theta = \left\{\theta_{\alpha} \in \mathbb{P}(\{0,1\}^{n_1})\; : \;\mathbb{P}_{\theta_{\alpha}}\left(\tilde{\alpha}_i = 1
    \right) = p_i, \;
   \;\mbox{for}\;   i \in [n_1] \right\}.
\end{array}
\end{equation*}
In this case, it is possible to compute the probabilities $\mathbb{P}_{\theta_{ind}}\big(\sum_{j=1}^{n_{2}} \tilde{\beta}_{j}= \ell\big), \; \ell \in [0, n_{2}]$ in polynomial time using dynamic programming recursion \cite{chen1997statistical}. We can then reformulate \eqref{eq:limitedtailtight} as follows:
 \begin{equation}
\begin{array}{lll}\label{eq:splitobjfn}
 \underset{\theta \in \Theta_{\ell}}{\max}\;\mathbb{P}_{\theta}\bigg(\displaystyle \sum_{i=1}^{n_{1}} \tilde \alpha_{i}+\sum_{j=1}^{n_{2}} \tilde\beta_{j}\geq r\bigg)
= \underset{\theta_{\alpha} \in \Theta}{\max}\;\displaystyle\sum_{\ell=0}^{n_{2}}\bigg[\mathbb{P}_{\theta_{\alpha}}\big(\displaystyle {\sum_{i=1}^{n_{1}} \tilde{\alpha}_{i}\geq r-\ell\big) \mathbb{P}_{\theta_{ind}}\big(\sum_{j=1}^{n_{2}} \tilde{\beta}_{j}= \ell}\big)\bigg].
\end{array}
\end{equation}
By rewriting the tail probabilities as:
$$\displaystyle \mathbb{P}_{ \theta_{\alpha}}\big( \sum_{i=1}^{n_{1}} \tilde{\alpha}_{i}\geq r-\ell \big)=\sum_{t=r-\ell}^{n_1} \mathbb{P}_{\theta_{\alpha}}\big(\sum_{i=1}^{n_{1}} \tilde{\alpha}_{i}=t\big),$$  we can cast \eqref{eq:splitobjfn} in the form of a weighted probability function similar to that in \eqref{opt:exposumweight} with $n_1$ decision variables and weights $w_{\ell}=  \mathbb{P}_{\theta_{ind}}\big(\sum_{j=1}^{n_{2}} \tilde{\beta}_{j}= \ell\big),\;\mbox{for}\; \ell \in[0,n_2]$. The compact linear program \eqref{opt:primalsumweightcompact} can now be used to compute the tight bound.

In this model, when $n = n_1$ and $n_2 = 0$, all the random variables are extremally dependent and the tight bound $S(r,1)$ is retrieved.
Similarly, when $n = n_2$ and $n_1=0$, all the random variables are mutually independent and the tail probability bound $I(r,1)$ is retrieved. Besides the other bounds, we also consider a Poisson approximation to sum of Bernoulli random variables. \cite{lecam1960poissonapproximation} showed that the Poisson distribution can be used to approximate the probability distribution of sums of independent but not necessarily identical Bernoulli random variables, where the error of the approximation is small when the probabilities are small.
The \cite{stein1972normal}-\cite{chen1975poisson} approximation method extends this idea and develops error bounds for the Poisson approximation of the distribution of sums weakly dependent Bernoulli variables. 
We compare the {limited dependency} bounds computed from the compact linear program \eqref{opt:primalsumweightcompact} with the two extremes of extremal dependence and complete independence and three other probabilities computed using a Poisson approximation, a comonotonic bound computed with perfectly dependent random variables and the Markov bound. Figure \ref{fig:splitn=30} shows the six bounds for $n=30$ variables where the {limited dependency} bounds (in purple) have been selectively shown for $n_1=6,10,14,18,22,26$ (left to right). In Figure \ref{fig:split30small}, we consider  non-identical small marginal probabilities by uniformly and independently generating the marginal probabilities between $0.1$ and $0.15$ while in Figure \ref{fig:split30rand}, we uniformly generate the probabilities in $ [0,1]$.
\begin{figure}[h!]
\begin{minipage}{0.48\textwidth}
\includegraphics[scale=0.38]{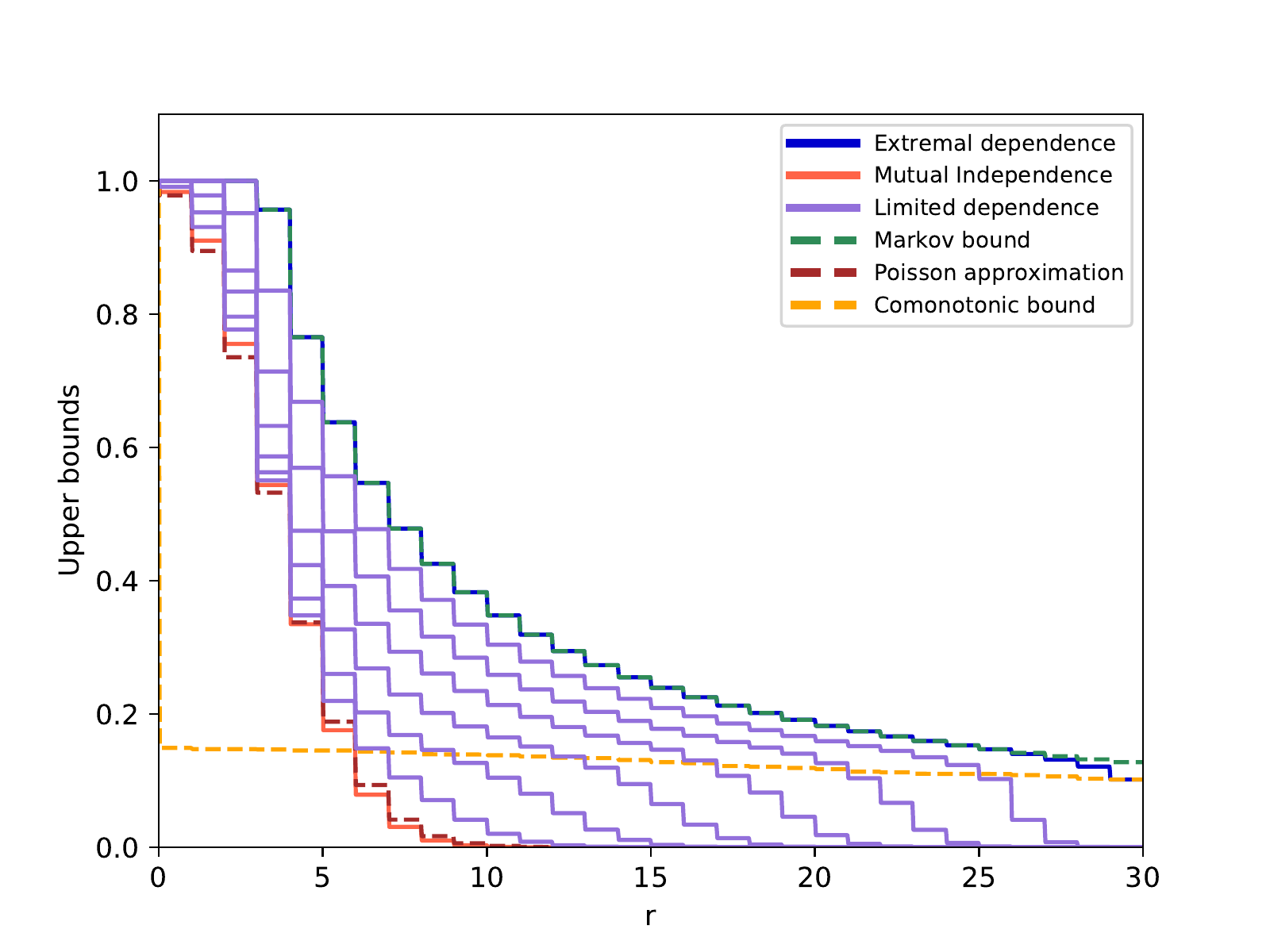}
\subcaption{Small range of marginal probabilities}
\label{fig:split30small}
\end{minipage}
\begin{minipage}{0.48\textwidth}
\includegraphics[scale=0.38]{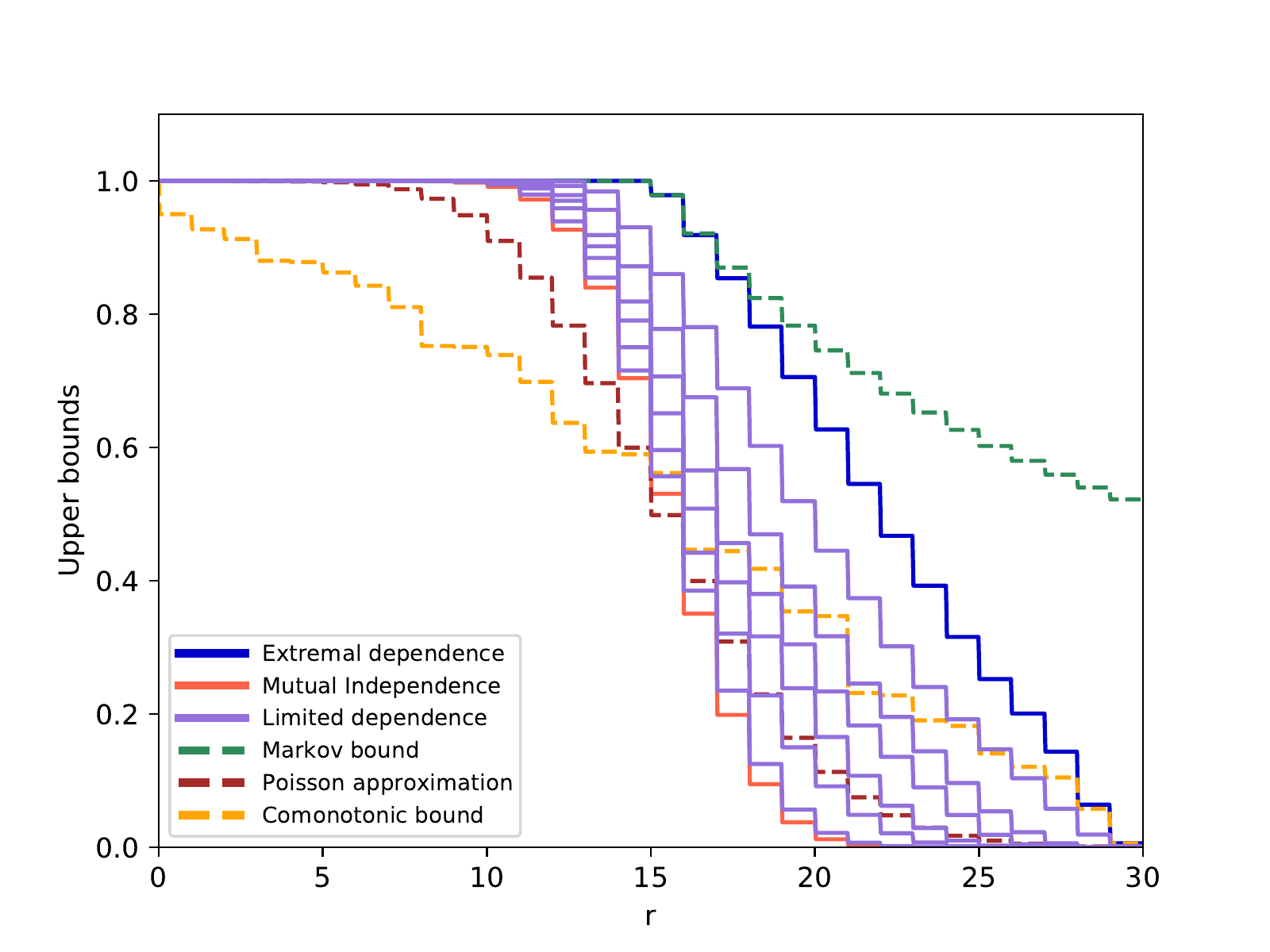}
\subcaption{Larger range of marginal probabilities}
\label{fig:split30rand}
\end{minipage}
\caption{Step plots of upper bounds for $n=30$}
\label{fig:splitn=30}
\end{figure}

The Poisson approximation closely follows the independent tail probability $I(r,1)$ in Figure \ref{fig:split30small} as the theory suggests with the assumption of small probabilities while in Figure \ref{fig:split30rand}, it initially underestimates the  independent tail probability (for $r \leq 15$) and then overestimates it. Due to the almost identical nature of the small probabilities in Figure \ref{fig:split30small}, the comonotonic bound plot remains almost flat for $r \geq 1$ and the Markov bound is very close to the extremally dependent bound $S(r,1)$ while this is not true in Figure \ref{fig:split30rand} due to the non-identical probabilities. The results indicate that the linear programming approach can appropriately incorporate both independence and dependence considerations in computing the extremal tail probability bounds.

\subsection{Random Walk: V-Polytope}
We now consider the maximum of partial sums of random variables, a problem arising from applications in random walks. Consider a random vector $\bold{\tilde{c}}$ of size  $n$ and let:
$$\displaystyle Z^{rw}(\bold{\tilde{c}}) = \max \left(\tilde{c}_1, \tilde{c}_1 + \tilde{c}_2, \ldots, \sum_{i=1}^n \tilde{c}_i\right),$$
where $\tilde{c}_i \in \{-1, 1\}$ for all $i \in [n]$. The tail behaviour of this quantity has been extensively studied (see \cite{Asmussen2001RuinP}) and is of interest in settings such as risk and queueing theory. For example, when $n \rightarrow \infty$ and the random variables are mutually independent, the Lundberg  inequality (see \cite{Asmussen1994}) gives the  tail probability bound,
$
\mathbb{P}_{\theta_{ind}}(Z^{rw}(\tilde{\bold{c}}) \geq r) \leq e^{-h_0r}
$,
where $h_0$ is parameter dependent on the moment generating function of the  distribution of $\tilde{\bold{c}}$.
 Several approximations for the distribution of $Z^{rw}(\tilde{\bold{c}})$ have been developed for the finite $n$ case (see \cite{Chung1948,korshunov1997distribution}) using the marginal distributions.  Here we consider the bounds on the tail probability with extremal dependence.

Let $U_{rw}(r)$ denote the maximum value of the tail probability over all joint distributions consistent with the given marginal distributions, $
U^{rw}(r) = \max_{\theta \in \Theta} \mathbb{P}(Z^{rw}(\tilde{\bold{c}}) \geq r)$.
 \Cref{fig:random_walk_identical} illustrates the probability bounds for the case of identical probabilities with $p_i = 0.5$ for all $i \in [n]$. `Tight UB' refers to the bound $U^{rw}(r)$. While the Markov bound applies to only non-negative random variables, in the random walk application considered, $Z^{rw}(\bold{c}) \in [-1, n]$. We therefore use the following variant,
\begin{align*}
\mathbb{P}(Z^{rw}(\tilde{\bold{c}})\geq r) = \mathbb{P}(Z^{rw}(\tilde{\bold{c}})+1 \geq r+1) \leq \min\left(\frac{ \max_{\theta \in \Theta} \mathbb{E}_{\theta}[Z^{rw}(\tilde{\bold{c}})] + 1}{r+1},1\right).
\end{align*}
We observe that the Markov bound is not a tight upper bound for this application. The probability bound `Worst exp' refers to a comonotone distribution here (since $Z^{rw}(\bold{c})$ is a supermodular function and the comonotone distribution maximizes expectation of supermodular functions) so that $\mathbb{P}(\tilde{c}_1 = 1, \ldots, \tilde{c}_n = 1) = 0.5$ and $\mathbb{P}(\tilde{c}_1 = -1, \ldots, \tilde{c}_n = -1) = 0.5$. The tight upper bound labelled `Tight UB' gives $U^{rw}(r)$ and is attained by a different distribution from the comonotone distribution. Similar trends are observed for the case of non-identical probabilities  in \Cref{fig:random_walk_non_identical}.
\begin{figure}[h!]
\begin{minipage}{0.45\textwidth}
\includegraphics[scale=0.43]{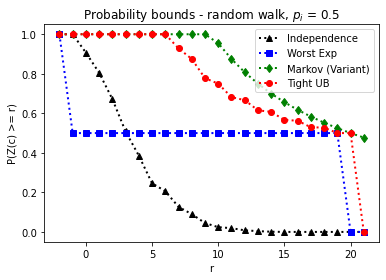}
\subcaption{The case of identical probabilities, $p_i = 0.5$.}
\label{fig:random_walk_identical}
\end{minipage}
\begin{minipage}{0.45\textwidth}
\includegraphics[scale=0.43]{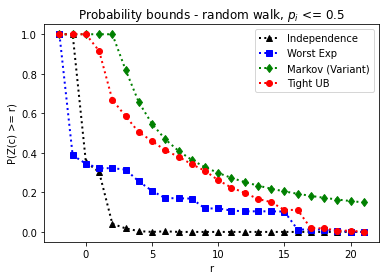}
\subcaption{A case of non-identical probabilities, $p_i \leq 0.5$.}
\label{fig:random_walk_non_identical}
\end{minipage}
\caption{Probability bounds for the random walk application.}
\label{fig:random_walk}
\end{figure}

\subsection{PERT Networks: H-Polytope}
We now discuss our numerical results in the context of PERT networks. We compute $U^{pert}(r)$ using the linear program in \Cref{thm:pert_primal_lp}. In the plots, this bound is denoted by the label `Tight UB'. The Markov bound is computed as $ \min(\max \mathbb{E}[Z(\tilde{\bold{c}})]/r,1)$ where the maximum  possible expectation bound is computed in polynomial time in the size of the graph using the below tight formulation from \cite{meilijson1979convex}.
\begin{align*}
\begin{array}{rllll}
 \max \mathbb{E}[Z(\tilde{\bold{c}})] = \min_{\bold{y}, \bold{d}, \bold{u}} & \displaystyle u_s + \sum_{(i,j) \in E} \sum_{k=1}^K p_{ijk} y_{ijk} \\
\mbox{s.t } &\displaystyle u_i - u_j \geq d_{ij},  \text{ for } (i,j) \in E, \\
&\displaystyle u_t = 0, \\
&\displaystyle y_{ijk} \geq k - d_{ij}, \text{ for } (i,j) \in E, \text{ for } k \in [0,K],\\
&\displaystyle \bold{y} \geq 0 , \bold{d}, \bold{u} \text{ unrestricted}.
\end{array}
\end{align*}
Formulation \eqref{opt:max_exp_z_c_minus_r} is used to obtain the tail probability from a distribution that maximizes $\mathbb{E}[Z(\tilde{\bold{c}}) -r]^+$, where $\mathcal{X}$ denotes the set of $s$-$t$ paths for PERT networks.

The network  in \Cref{fig:example4} with $ n  = 24$ nodes and a total of $29$ edges or activities is considered. There are a total of $14$ paths from $s$ to $t$. The longest path from $s$ to $t$ contains $10$ edges and hence the maximum possible completion time of the project is $10K$, where $K$ is the maximum possible duration of each of the activities. This network was presented in \cite{Birge1995,Bertsimas2004} where the worst case bounds for the expected time of completion was computed.  In the examples we consider, for all  edges $(i,j)$, the probability $p_{ijk} = 1/(K+1)$. We take $K=10$.

\begin{minipage}{0.48\linewidth}
\begin{tikzpicture}[node distance=0.6cm,
  inner/.style={circle,draw=black,thick,inner sep=0.6 pt},  ]
  \node[inner](s) {s};
      \node [inner, above right=0.5cm of s] (1)  {1};
      \node [inner, right=0.5cm of s] (2) {2};
      \node [inner,below right=0.5cm of s] (3) {3};
       \node [inner,  right=of 3] (4)  {4};
            \node [inner,  right=of 4] (5)  {5};
                   \node [inner,  right=0.6cm of 5] (6)  {6};
               \node [inner,  right=0.6cm of 6] (7)  {7};
           \node [inner,  right=0.6cm of 7] (8)  {8};
        \node [inner, above right=0.5cm of 1] (9)  {9};
         \node [inner,  right=0.4cm of 9] (10)  {10};
           \node [inner, above right=of 9] (16)  {16};
            \node [inner,  above left=0.4 cm of 5] (11)  {11};
                   \node [inner,  right=of 11] (12)  {12};
                    \node [inner,  right=of 12] (13)  {13};
               \node [inner,   right=of 10] (14)  {14};
            \node [inner,  right=of 14] (15)  {15};
      \node [inner,  right=of 16] (17)  {17};
     \node [inner,  right=of 17] (18)  {18};
     \node[inner, below right= of 10] (19) {19};
      \node [inner,  right=of 19] (20)  {20};
       \node [inner,  right=of 20] (21)  {21};
        \node [inner,  right=0.3cm of 21] (22)  {22};
         \node [inner,  below=0.3cm of 22] (t)  {t};

\draw[thick,->] (s) -- (1);
\draw[thick,->] (s) -- (2);
\draw[thick,->] (s) -- (3);
\draw[thick,->] (1) -- (9);
\draw[thick,->] (2) -- (4);
\draw[thick,->] (3) -- (4);
\draw[thick,->] (4) -- (5);
\draw[thick,->] (4) -- (9);
\draw[thick,->] (5) -- (6);
\draw[thick,->] (6) -- (7);
\draw[thick,->] (7) -- (8);
\draw[thick,->] (8) -- (13);
\draw[thick,->] (9) -- (10);
\draw[thick,->] (9) -- (16);
\draw[thick,->] (10) -- (11);
\draw[thick,->] (10) -- (14);
\draw[thick,->] (10) -- (19);
\draw[thick,->] (11) -- (12);
\draw[thick,->] (12) -- (13);
\draw[thick,->] (13) -- (21);
\draw[thick,->] (14) -- (15);
\draw[thick,->] (15) -- (21);
\draw[thick,->] (16) -- (17);
\draw[thick,->] (17) -- (18);
\draw[thick,->] (18) -- (21);
\draw[thick,->] (19) -- (20);
\draw[thick,->] (20) -- (21);
\draw[thick,->] (21) -- (22);
\draw[thick,->] (22) -- (t);
\end{tikzpicture}
\captionof{figure}{Example 3 }
\label{fig:example4}
\end{minipage}
\begin{minipage}{0.45\linewidth}
\includegraphics[scale=0.43]{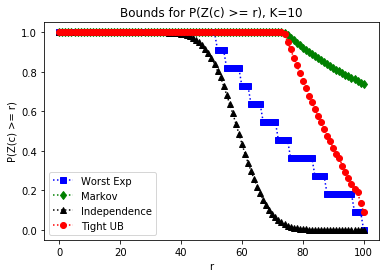}
\end{minipage}

The Markov bound is not tight for this example while the gaps from independence and worst exp demonstrate significant gap with the tight bound. Here, the worst exp curve is closer to Tight UB than independence. However the distribution maximizing the worst case expectation does not maximize the tail probability.

\subsubsection{Comparison of Bounds on Randomly Generated Instances}
We now compare our bounds against the Markov bound and the bound from the independent distribution for a set of $50$ randomly generated graphs and univariate marginals on $n= 10$ nodes with $K=10$.  In \Cref{fig:gaps_markov}, we report the gap $M(r) - U^{pert}(r)$  for various values of $r$ where $M(r)$ represents the Markov bound. The bars indicate the range between the minimum and maximum gaps while the dotted line provides the mean gap. Observe that the Markov bounds are not tight in general and always provide an upper bound for $U^{pert}(.)$. In \Cref{fig:gaps_independence}, we  report the gap $ U^{pert}(r) - \mathbb{P}_{\theta_{ind}} (Z(\tilde{\bold{c}} \geq r)$ where $\theta_{ind}$ denotes the independent distribution. The independent distribution serves as lower bound for $U^{pert}(r)$ and is clearly not an extremal distribution.
\begin{figure}[h!]
\begin{minipage}{0.48\textwidth}
\includegraphics[scale=0.45]{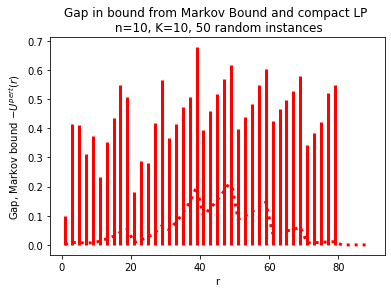}
\subcaption{Gap in Markov bound}
\label{fig:gaps_markov}
\end{minipage}
\begin{minipage}{0.48\textwidth}
\includegraphics[scale=0.45]{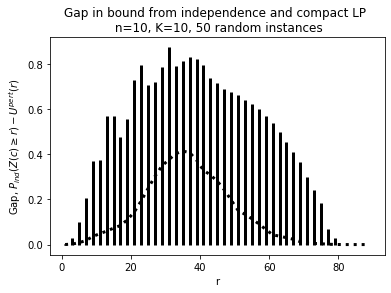}
\subcaption{Gap in independence}
\label{fig:gaps_independence}
\end{minipage}
\caption{Comparison of gaps in various bounds over 50 randomly generated instances. }
\label{fig:bd_comparison}
\end{figure}

\subsubsection{Computational Times}
We now report the computational times of our compact linear program as a function of the number of nodes $n$ as well as a function of $K$. \Cref{fig:times_with_n} shows the error bars of the execution time as a function of $n$, over $50$ random instances with $r=40$ and $K$ fixed to $10$.  Even for $n=100$ nodes, the execution time is about $1.2$ seconds on an average.  We performed the experiment for various values of $r \in \{10, \ldots, 50 \}$, however we did not observe significant difference in the results. In \Cref{fig:times_with_k}, we  provide the error bars of  the execution time as a function of $K$, with $r=50$ and $n = 20$.  Over all instances, our compact LP takes a maximum of $0.45$ seconds even when the support for the activity durations goes till $K=100$.
\begin{figure}[h!]
\begin{minipage}{0.48\textwidth}
\includegraphics[scale=0.43]{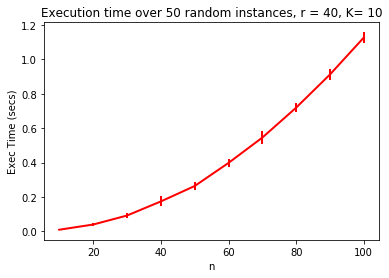}
\subcaption{}
\label{fig:times_with_n}
\end{minipage}
\begin{minipage}{0.48\textwidth}
\includegraphics[scale=0.43]{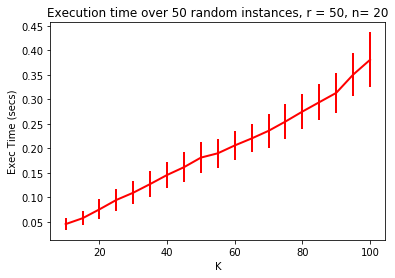}
\subcaption{}
\label{fig:times_with_k}
\end{minipage}
\caption{Execution times of our compact linear program}
\label{fig:execution_time}
\end{figure}
\vspace{-1cm}

\section*{Acknowledgements}
The research of the fourth author was partly supported by MOE Academic Research Fund Tier 2 grant T2MOE1906, ``Enhancing Robustness of Networks to Dependence via Optimization''.
\bibliographystyle{siamplain}
\bibliography{prob_bounds_lps_v3}
\end{document}

%% file: ex_shared.tex
\usepackage{lipsum}
\usepackage{amsfonts}
\usepackage{graphicx}
\usepackage{epstopdf}
\usepackage{algorithmic}
\ifpdf
  \DeclareGraphicsExtensions{.eps,.pdf,.png,.jpg}
\else
  \DeclareGraphicsExtensions{.eps}
\fi

\newsiamremark{remark}{Remark}
\newsiamremark{hypothesis}{Hypothesis}
\crefname{hypothesis}{Hypothesis}{Hypotheses}
\newsiamthm{claim}{Claim}

\headers{Extremal Probability Bounds}{D. Padmanabhan, S. D. Ahipasaoglu, A. Ramachandra and K. Natarajan}

\title{Extremal Probability Bounds in Combinatorial Optimization\thanks{Submitted: August 2021}}

\author{Divya Padmanabhan\thanks{School of Mathematics and Computer Science, Indian Institute of Technology, Ponda-403401, Goa. Email: divya@iitgoa.ac.in}
\and
Selin Damla Ahipasaoglu\thanks{Mathematical Sciences, University of Southampton, Highfield Southampton SO17 1BJ. Email: sda1u20@soton.ac.uk}  \and
Arjun Ramachandra\thanks{Engineering Systems and Design, Singapore University of Technology and Design, 8 Somapah Road, Singapore 487372. Email: arjun\_ramachandra@mymail.sutd.edu.sg}  \and
Karthik Natarajan\thanks{Engineering Systems and Design, Singapore University of Technology and Design, 8 Somapah Road, Singapore 487372. Email: karthik\_natarajan@sutd.edu.sg}}

\usepackage{amsopn}
